\documentclass[10pt,reqno,oneside]{amsart} \usepackage{graphicx,verbatim}  \usepackage{amssymb,cite} \usepackage{epstopdf} \usepackage{graphicx} \usepackage{amsmath} \usepackage{amsthm, enumerate} \usepackage{mathrsfs} \usepackage{caption} \usepackage{subcaption} \usepackage{xcolor} \usepackage[normalem]{ulem}    \allowdisplaybreaks   \chardef\forshowkeys=0   \chardef\showllabel=0   \chardef\refcheck=0   \chardef\sketches=0 \usepackage{float} \floatstyle{boxed} \usepackage{marginnote} \usepackage[colorlinks=true, pdfstartview=FitV, linkcolor=blue, citecolor=blue, urlcolor=blue]{hyperref} \ifnum\forshowkeys=1      \usepackage[notref,notcite,color]{showkeys} \fi \ifnum\showllabel=1  \def\llabel#1{\marginnote{\color{lightgray}\rm\small(#1)}[-0.0cm]\notag} \else  \def\llabel#1{\notag} \fi \title{Asymptotic properties of the Boussinesq Equations with Dirichlet Boundary Conditions} \author[I.~Kukavica]{Igor Kukavica} \address{Department of Mathematics\\ University of Southern California\\ Los Angeles, CA 90089} \email{kukavica@usc.edu} \author[D.~Massatt]{David Massatt} \address{Department of Mathematics\\ University of Southern California\\ Los Angeles, CA 90089} \email{dmassatt@usc.edu} \author[M.~Ziane]{Mohammed Ziane} \address{Department of Mathematics\\ University of Southern California\\ Los Angeles, CA 90089} \email{ziane@usc.edu} \numberwithin{equation}{section}   \newtheorem{thm}{Theorem}[subsection]  \newtheorem{lemma}{Lemma}    \numberwithin{definition}{section} \numberwithin{thm}{section} \numberwithin{remark}{section} \numberwithin{prop}{section} \numberwithin{corollary}{section} \numberwithin{lemma}{section} \numberwithin{assumption}{section} \textwidth 16.4truecm \textheight 8in\oddsidemargin0.2truecm\evensidemargin0.7truecm\voffset-0.5truecm \def\tepsilon{{\tilde\epsilon}}  \def\uone{u^{(1)}} \def\utwo{u^{(2)}} \def\thetaone{\theta^{(1)}} \def\thetatwo{\theta^{(2)}} \def\rhoone{\rho^{(1)}} \def\rhotwo{\rho^{(2)}} \def\inon#1{\qquad{}\quad{}\hbox{#1}}                 \def\as#1{\qquad{}\quad{}\hbox{as #1}}                 \def\un{u^{(n)}} \def\unp{u^{(n+1)}}      \def\Pnp{P^{(n+1)}}      \def\thetanp{\theta^{(n+1)}}    \def\tzeta{\tilde\zeta} \def\asti{\text{\,\,\,\,\,\,as~$t\to\infty$}}  \definecolor{colororange}{rgb}{0.8,0.2,0} \definecolor{colorpurple}{rgb}{0.6,0.0,0.6}    \def\colb{\color{black}}   \def\indeq{\qquad{}} \def\startnewsection#1#2{\section{#1}\label{#2}\setcounter{equation}{0}\color{black}}               \def\curl{\mathop{\rm curl}\nolimits}   \def\dist{\mathop{\rm dist}\nolimits}   \def\supp{\mathop{\rm supp}\nolimits} \definecolor{colorgggg}{rgb}{0.1,0.5,0.3} \definecolor{colorllll}{rgb}{0.0,0.7,0.0} \definecolor{colorhhhh}{rgb}{0.5,0.0,0.3} \definecolor{colorpppp}{rgb}{0.7,0.0,0.2} \definecolor{coloroooo}{rgb}{0.9,0.4,0} \definecolor{colorqqqq}{rgb}{0.1,0.7,0}   \def\cole{\color{black}}       \def\comma{ {\rm ,\qquad{}} }                        \def\les{\lesssim}                \newcommand{\Laplace}{\Delta{}} \newcommand{\loc}{\text{loc}} \newcommand{\grad}{\nabla} \ifnum\refcheck=1   \usepackage{refcheck} \fi \begin{document} \begin{abstract} We address the asymptotic properties for the Boussinesq equations with vanishing thermal diffusivity in a bounded domain with no-slip boundary conditions. We show the dissipation of the $L^2$ norm of the velocity and its gradient, convergence of the $L^2$ norm of $Au$, and an $o(1)$-type  exponential growth for $\Vert A^{3/2}u\Vert_{L^2}$. We also obtain that in the interior of the domain the gradient of the vorticity is bounded by a polynomial function of time. \end{abstract} \baselineskip=.5truecm \date{\today}    \maketitle \tableofcontents \startnewsection{Introduction}{sec01} In this paper, we address the asymptotic behavior of the Boussinesq equations   \begin{align}   \begin{split}   & u_t - \nu \Laplace u + u \cdot \grad u + \grad p = \rho e_2    \\&   \rho_t + u \cdot \grad \rho = 0    \\&   \grad \cdot u = 0   \end{split}    \label{EWRTDFBSERSDFGSDFHDSFADSFSDFSVBASWERTDFSGSDFGFDGASFASDFDSFHFDGHFH58}   \end{align} with vanishing thermal/density diffusivity, in a smooth bounded domain $\Omega \subseteq \mathbb{R}^2$ with the Dirichlet boundary condition   \begin{align}   \begin{split}   u\bigl{|}_{\partial \Omega} = 0   \end{split}    \label{EWRTDFBSERSDFGSDFHDSFADSFSDFSVBASWERTDFSGSDFGFDGASFASDFDSFHFDGHFH126}   \end{align} and subject to the initial condition $(u(0),\rho(0))=(u_0 , \rho_0)$. Here, $u$ represents the velocity, $p$ the pressure, and $\rho$ the density or the temperature, depending on the physical context.  The 2D Boussinesq system of equations is used in a wide range of physical contexts, from large scale oceanic and atmospheric flows where rotation and stratification are significant to microfluids and biophysics. It also relates closely to fundamental models in fluid dynamics. In particular, the vorticity formulation of the incompressible Euler equations away from the singularity can be described by the 2D Boussinesq equations (cf.~\cite{DWZZ}). For simplicity of exposition, we shall  refer to the variable $\rho$ as the density, although it may also represent a temperature. \par While global existence results have been well-known in the case of positive viscosity and positive thermal diffusivity, i.e., when adding the term $-\kappa \Delta \rho$ in the equation for the density/temperature, we address here the case of vanishing thermal diffusivity. In the case when both viscosity $\nu$ and diffusion coefficients $\kappa$ vanish, the global existence and uniqueness remain open questions, although results on the local existence, blow-up criteria, explicit solutions, and finite time singularities have been proven; cf.~the blow-up results in \cite{CH,EJ}, based on the singularity creation theorem for the Euler equations by Elgindi~\cite{E}. The case $\nu>0$ and $\kappa=0$, considered here,  was initially considered by Chae~\cite{C} and Hou and Li~\cite{HL}. In particular, Hou and Li obtained the global existence and persistence of regularity in $H^{s}\times H^{s-1}$ for integer valued $s\geq3$ in the case of periodic boundary conditions. The paper \cite{LLT} by Lai~et~al extended the result in \cite{HL} to the Dirichlet boundary conditions. The persistence of regularity for the lower value $s=2$ in the case of Dirichlet or periodic boundary conditions was addressed in \cite{HKZ1}. 
Subsequently, Ju obtained in \cite{J} that $C e^{Ct^2}$ is an upper bound for the $H^{1}$ norm for the density, also for the Dirichlet boundary conditions. The bound was lowered to $e^{Ct}$ in \cite{KW2}, where also more precise results were obtained for periodic boundary conditions. In particular, \cite[Theorem~2.1]{KW2} contains a uniform in time upper bound for the quantity $ \Vert D^{2}u\Vert_{L^p}$ for all $p\geq2$ in the periodic case. In a recent paper by Doering~et~al~\cite{DWZZ}, the global existence, uniqueness, and regularity for the Boussinesq for the Lions boundary condition on a Lipschitz domain $\Omega$, was proven along with the dissipation of the $L^2$ norm of the velocity and its gradient. For other papers on the global existence and the regularity in Sobolev and Besov spaces, see \cite{ACW,ACSetal, BFL, BS, BrS,CD,CG,CN,CW,DP,HK1,HK2,HKR,HKZ2,HS,JMWZ,KTW,KW2,KWZ,LPZ,SW}. \colb \par In this paper, we prove several results on the asymptotic  behavior of solutions of the Boussinesq system \eqref{EWRTDFBSERSDFGSDFHDSFADSFSDFSVBASWERTDFSGSDFGFDGASFASDFDSFHFDGHFH58} with the Dirichlet boundary conditions \eqref{EWRTDFBSERSDFGSDFHDSFADSFSDFSVBASWERTDFSGSDFGFDGASFASDFDSFHFDGHFH126}. In our first main theorem, Theorem~\ref{T01},  we show that the $H^1$ norm of the velocity dissipates.  We also establish a balanced convergence of $Au$, cf.~\eqref{EWRTDFBSERSDFGSDFHDSFADSFSDFSVBASWERTDFSGSDFGFDGASFASDFDSFHFDGHFH10} below, where $A$ is the Stokes operator. Regarding the growth of the density, we prove that the first Sobolev norm of the density is bounded, up to a constant, by $e^{\epsilon t}$ for an arbitrarily small $\epsilon>0$, thus improving a result from \cite{KW2} where the bound of the type $e^{C t}$ was proven.  Since the growth of the Sobolev norms of the density is controlled by the time integral of $\Vert \nabla u\Vert_{L^\infty}$, it is reasonable to expect that the bound was optimal; however, here we prove that the optimal bound is in fact $e^{\epsilon t}$. It remains an open problem if one can achieve the estimate of the type $e^{C t^{\alpha}}$, where $\alpha\in[0,1)$. The theorem holds under the assumption that $(u_0,\rho_0)$ belongs to $H^2\times H^1$.  The ideas for the proof of Theorems~\ref{T01} draw from the approaches in \cite{DWZZ}, \cite{HKZ1}, \cite{LLT}, \cite{HKZ1}, \cite{J}, \cite{KW1}, and \cite{KW2}. Additionally, in Theorem~\ref{T03}, we show that the theorem and the persistence of regularity also hold under the $H^1\times H^{1}$ assumption on the initial data. \par In the second main theorem, Theorem~\ref{T02}, we address the behavior of the solution in a higher regularity norm. We prove that, under the $H^{3}\times H^{2}$ assumption on the initial data, that  for every $\epsilon>0$ the norm of $(u,\rho)$ in the $H^{3}\times H^{2}$ norm is bounded by $e^{\epsilon t}$, up to a constant depending on $\epsilon>0$. This holds under the $H^{3}\times H^{2}$ regularity of the initial data $(u_0,\rho_0)$. We point out that, as in Theorem~\ref{T03}, the same in fact holds under the $H^2\times H^2$ assumption on the data. \par In the last main theorem, Theorem~\ref{T04},  we consider the upper bound for the $L^p$ norm of the second derivatives of the velocity. As shown in \cite{HKZ1}, one may obtain a uniform bound when $p=2$. When $p>2$, this is not known except in the case of periodic boundary condition, which is a result obtained in \cite{KW1}. Here, we prove that we can obtain a polynomial in time bound in the interior of a domain when considering the Dirichlet boundary condition, which is considerably lower than $e^{\epsilon t}$ type bound that would result from applying the Gagliardo-Sobolev inequality on the conclusions of Theorem~\ref{T02}. The proof is obtained by the change of variable from \cite{KW1} combined with new localization arguments controlling the nonlocal nature of the transformation in \cite{KW1} (see the double cut-off strategy in the proof of Theorem~\ref{T04} below). \par We emphasize that all our results extend also in the often-studied problem of the channel with Dirichlet boundary conditions on top and the bottom and periodic boundary conditions on the sides. Also, our proofs are completely self-contained. \par \startnewsection{Main theorems}{sec02} We consider the asymptotic behavior of the Boussinesq equations   \begin{align}   \begin{split}   & u_t - \Laplace u + u \cdot \grad u + \grad p = \rho e_2    \\&   \rho_t + u \cdot \grad \rho = 0    \\&   \grad \cdot u = 0   \end{split}   \label{EWRTDFBSERSDFGSDFHDSFADSFSDFSVBASWERTDFSGSDFGFDGASFASDFDSFHFDGHFH01}   \end{align} and   \begin{align}   \begin{split}   u\bigl{|}_{\partial \Omega} = 0   ,   \end{split}   \label{EWRTDFBSERSDFGSDFHDSFADSFSDFSVBASWERTDFSGSDFGFDGASFASDFDSFHFDGHFH02}   \end{align} coupling the Navier-Stokes equations \cite{CF,DG,K1,K2,R,T1,T2,T3} for the velocity $u=(u_1,u_2)$ and the pressure $p$ with the equation for the density~$\rho$. The system is set on a smooth, bounded, and connected domain $\Omega \subseteq \mathbb{R}^2$ and supplemented with the initial condition   \begin{equation}    (u,\rho)(0)    =    (u_0,\rho_0)    \inon{in~$\Omega$}    .    \llabel{8Th sw ELzX U3X7 Ebd1Kd Z7 v 1rN 3Gi irR XG KWK0 99ov BM0FDJ Cv k opY NQ2 aN9 4Z 7k0U nUKa mE3OjU 8D F YFF okb SI2 J9 V9gV lM8A LWThDP nP u 3EL 7HP D2V Da ZTgg zcCC mbvc70 qq P cC9 mt6 0og cr TiA3 HEjw TK8ymK eu J Mc4 q6d Vz2 00 XnYU tLR9 GYjPXv FO V r6W 1zU K1W bP ToaW JJuK nxBLnd 0f t DEb Mmj 4lo HY yhZy MjM9 1zQS4p 7z 8 eKa 9h0 Jrb ac ekcEWRTDFBSERSDFGSDFHDSFADSFSDFSVBASWERTDFSGSDFGFDGASFASDFDSFHFDGHFH03}   \end{equation} Here, $u$ denotes the velocity, $p$ the pressure, and $\rho$ the density. Note that we set $\nu = 1$ for simplicity of exposition; all the results extend  to other values of~$\nu$ with  \def\FGSDFHGFHDFGHDFGH{\partial}constants depending additionally on~$\nu$. \par From \cite{CF,T1}, we recall the classical spaces    \begin{align}    H = \{u \in L^2(\Omega): \grad \cdot u = 0 \text{~in~$\Omega$}, u \cdot n = 0 \text{~on~$\FGSDFHGFHDFGHDFGH\Omega$}\}    ,    \llabel{i rexG 0z4n3x z0 Q OWS vFj 3jL hW XUIU 21iI AwJtI3 Rb W a90 I7r zAI qI 3UEl UJG7 tLtUXz w4 K QNE TvX zqW au jEMe nYlN IzLGxg B3 A uJ8 6VS 6Rc PJ 8OXW w8im tcKZEz Ho p 84G 1gS As0 PC owMI 2fLK TdD60y nH g 7lk NFj JLq Oo Qvfk fZBN G3o1Dg Cn 9 hyU h5V SP5 z6 1qvQ wceU dVJJsB vX D G4E LHQ HIa PT bMTr sLsm tXGyOB 7p 2 Os4 3US bq5 ik 4Lin 769O TkUEWRTDFBSERSDFGSDFHDSFADSFSDFSVBASWERTDFSGSDFGFDGASFASDFDSFHFDGHFH04}   \end{align} where $n$ denotes the outward unit normal, and   \begin{align}    V = \{u \in H_0^1(\Omega): \grad \cdot u = 0 \text{~in~$\Omega$} \}    ,    \llabel{xmp I8 u GYn fBK bYI 9A QzCF w3h0 geJftZ ZK U 74r Yle ajm km ZJdi TGHO OaSt1N nl B 7Y7 h0y oWJ ry rVrT zHO8 2S7oub QA W x9d z2X YWB e5 Kf3A LsUF vqgtM2 O2 I dim rjZ 7RN 28 4KGY trVa WW4nTZ XV b RVo Q77 hVL X6 K2kq FWFm aZnsF9 Ch p 8Kx rsc SGP iS tVXB J3xZ cD5IP4 Fu 9 Lcd TR2 Vwb cL DlGK 1ro3 EEyqEA zw 6 sKe Eg2 sFf jz MtrZ 9kbd xNw66c xf t lEWRTDFBSERSDFGSDFHDSFADSFSDFSVBASWERTDFSGSDFGFDGASFASDFDSFHFDGHFH05}    \end{align} utilized in the study of the Navier-Stokes equations. With $\mathbb{P}\colon L^2 \to H$ the Leray projector, denote by   \begin{align}    A = -\mathbb{P}\Laplace    ,    \llabel{zD GZh xQA WQ KkSX jqmm rEpNuG 6P y loq 8hH lSf Ma LXm5 RzEX W4Y1Bq ib 3 UOh Yw9 5h6 f6 o8kw 6frZ wg6fIy XP n ae1 TQJ Mt2 TT fWWf jJrX ilpYGr Ul Q 4uM 7Ds p0r Vg 3gIE mQOz TFh9LA KO 8 csQ u6m h25 r8 WqRI DZWg SYkWDu lL 8 Gpt ZW1 0Gd SY FUXL zyQZ hVZMn9 am P 9aE Wzk au0 6d ZghM ym3R jfdePG ln 8 s7x HYC IV9 Hw Ka6v EjH5 J8Ipr7 Nk C xWR 84T WnqEWRTDFBSERSDFGSDFHDSFADSFSDFSVBASWERTDFSGSDFGFDGASFASDFDSFHFDGHFH06}   \end{align} the Stokes operator with the domain $D(A) = H^2(\Omega) \cap V$. \par It is known that  for a sufficiently regular initial condition there exists a unique, global in time solution for \eqref{EWRTDFBSERSDFGSDFHDSFADSFSDFSVBASWERTDFSGSDFGFDGASFASDFDSFHFDGHFH01}--\eqref{EWRTDFBSERSDFGSDFHDSFADSFSDFSVBASWERTDFSGSDFGFDGASFASDFDSFHFDGHFH02} (cf.~\cite{C,HL}). In the first theorem, we obtain the asymptotic properties of  $A^{1/2}u$ and $A u$ in the energy norm. \par \cole \begin{thm} \label{T01} Let $(u_0 , \rho_0) \in (H^2(\Omega)\cap V) \times H^1(\Omega)$. Then the solution    \begin{equation}    (u,\rho)\in ( C([0,\infty);H) \cap L_{\loc}^2([0,\infty);D(A)) ) \times L^{\infty}_{\loc}([0,\infty),H^1(\Omega))       \llabel{ s0 fsiP qGgs Id1fs5 3A T 71q RIc zPX 77 Si23 GirL 9MQZ4F pi g dru NYt h1K 4M Zilv rRk6 B4W5B8 Id 3 Xq9 nhx EN4 P6 ipZl a2UQ Qx8mda g7 r VD3 zdD rhB vk LDJo tKyV 5IrmyJ R5 e txS 1cv EsY xG zj2T rfSR myZo4L m5 D mqN iZd acg GQ 0KRw QKGX g9o8v8 wm B fUu tCO cKc zz kx4U fhuA a8pYzW Vq 9 Sp6 CmA cZL Mx ceBX Dwug sjWuii Gl v JDb 08h BOV C1 pni6 4EWRTDFBSERSDFGSDFHDSFADSFSDFSVBASWERTDFSGSDFGFDGASFASDFDSFHFDGHFH07}   \end{equation} of \eqref{EWRTDFBSERSDFGSDFHDSFADSFSDFSVBASWERTDFSGSDFGFDGASFASDFDSFHFDGHFH01}--\eqref{EWRTDFBSERSDFGSDFHDSFADSFSDFSVBASWERTDFSGSDFGFDGASFASDFDSFHFDGHFH02} satisfies   \begin{equation}    \Vert Au\Vert_{L^2}\leq C       ,    \label{EWRTDFBSERSDFGSDFHDSFADSFSDFSVBASWERTDFSGSDFGFDGASFASDFDSFHFDGHFH08}   \end{equation} where $C$ depends on the size of the initial data, i.e., on the  norms $\Vert A u_0\Vert_{L^2}$ and $\Vert \rho_0\Vert_{H^1}$. Moreover,   \begin{align}   \begin{split}   \Vert A^{1/2}u\Vert_{L^2}   =   \Vert \grad u \Vert_{L^2} \to 0   \asti   ,   \end{split}    \label{EWRTDFBSERSDFGSDFHDSFADSFSDFSVBASWERTDFSGSDFGFDGASFASDFDSFHFDGHFH09}   \end{align} and   \begin{align}   \begin{split}   \Vert A u - \mathbb{P}(\rho e_2) \Vert_{L^2} \to 0   \asti   ,   \end{split}    \label{EWRTDFBSERSDFGSDFHDSFADSFSDFSVBASWERTDFSGSDFGFDGASFASDFDSFHFDGHFH10}     \end{align} and for every $\epsilon >0$ we have   \begin{align}   \begin{split}   \Vert \rho(t) \Vert_{H^1} \leq C_\epsilon e^{\epsilon t}    \comma t\geq 0    ,   \end{split}    \label{EWRTDFBSERSDFGSDFHDSFADSFSDFSVBASWERTDFSGSDFGFDGASFASDFDSFHFDGHFH11}   \end{align} where $C_\epsilon$ is a constant depending on $\epsilon$ and the size of initial data. \end{thm} \colb \par Above and in the sequel, we allow all constants to depend on~$\Omega$. We note that in Theorem~\ref{T01} the assumption of $H^2$ regularity on the initial velocity can be relaxed to $u_0\in V$, as shown in Theorem~\ref{T03} below. In the next statement, we obtain the asymptotic behavior of the $H^{3}\times H^{2}$ norm of the solution~$(u,\rho)$. From \cite{LLT,T5}, the local existence requires the initial data to satisfy the compatibility condition    \begin{align}      \begin{split}     (-\Laplace u_0  - \grad p_0- \rho_0e_2)|_{\FGSDFHGFHDFGHDFGH\Omega} = 0    ,   \end{split}    \label{EWRTDFBSERSDFGSDFHDSFADSFSDFSVBASWERTDFSGSDFGFDGASFASDFDSFHFDGHFH12}    \end{align} where $p_0$ denotes the initial pressure, which solves the Neumann boundary problem   \begin{align}   \begin{split}   &\Laplace p_0 = \grad \cdot (\rho_0e_2 - u_0 \cdot \grad u_0)    \inon{in $\FGSDFHGFHDFGHDFGH\Omega$}   \\&   \grad p_0 \cdot n\bigl{|}_{\FGSDFHGFHDFGHDFGH \Omega} = (\Laplace u_0 + \rho_0e_2) \cdot n\bigl{|}_{\FGSDFHGFHDFGHDFGH \Omega}   \end{split}   \llabel{TTq Opzezq ZB J y5o KS8 BhH sd nKkH gnZl UCm7j0 Iv Y jQE 7JN 9fd ED ddys 3y1x 52pbiG Lc a 71j G3e uli Ce uzv2 R40Q 50JZUB uK d U3m May 0uo S7 ulWD h7qG 2FKw2T JX z BES 2Jk Q4U Dy 4aJ2 IXs4 RNH41s py T GNh hk0 w5Z C8 B3nU Bp9p 8eLKh8 UO 4 fMq Y6w lcA GM xCHt vlOx MqAJoQ QU 1 e8a 2aX 9Y6 2r lIS6 dejK Y3KCUm 25 7 oCl VeE e8p 1z UJSv bmLd Fy7ObQEWRTDFBSERSDFGSDFHDSFADSFSDFSVBASWERTDFSGSDFGFDGASFASDFDSFHFDGHFH12b}   \end{align} with $n$ denoting the outward unit normal. \par \cole \begin{thm} \label{T02} Assume that $(u_0 , \rho_0) \in (H^{3}(\Omega) \cap V) \times H^2(\Omega)$ satisfies the compatibility condition \eqref{EWRTDFBSERSDFGSDFHDSFADSFSDFSVBASWERTDFSGSDFGFDGASFASDFDSFHFDGHFH12}, and let $(u,\rho)$ be the corresponding solution of \eqref{EWRTDFBSERSDFGSDFHDSFADSFSDFSVBASWERTDFSGSDFGFDGASFASDFDSFHFDGHFH01}--\eqref{EWRTDFBSERSDFGSDFHDSFADSFSDFSVBASWERTDFSGSDFGFDGASFASDFDSFHFDGHFH02}. Then for every $\epsilon>0$, we have   \begin{align}   \begin{split}   \Vert u(t) \Vert_{H^{3}} \leq C_\epsilon e^{\epsilon t}   \end{split}    \comma t\geq0    \llabel{ FN l J6F RdF kEm qM N0Fd NZJ0 8DYuq2 pL X JNz 4rO ZkZ X2 IjTD 1fVt z4BmFI Pi 0 GKD R2W PhO zH zTLP lbAE OT9XW0 gb T Lb3 XRQ qGG 8o 4TPE 6WRc uMqMXh s6 x Ofv 8st jDi u8 rtJt TKSK jlGkGw t8 n FDx jA9 fCm iu FqMW jeox 5Akw3w Sd 8 1vK 8c4 C0O dj CHIs eHUO hyqGx3 Kw O lDq l1Y 4NY 4I vI7X DE4c FeXdFV bC F HaJ sb4 OC0 hu Mj65 J4fa vgGo7q Y5 X tLy EWRTDFBSERSDFGSDFHDSFADSFSDFSVBASWERTDFSGSDFGFDGASFASDFDSFHFDGHFH13}   \end{align} and    \begin{align}   \begin{split}   \Vert \rho(t) \Vert_{H^2} \leq C_\epsilon e^{\epsilon t}    \comma t\geq0   ,   \end{split}    \label{EWRTDFBSERSDFGSDFHDSFADSFSDFSVBASWERTDFSGSDFGFDGASFASDFDSFHFDGHFH14}   \end{align}  where $C_{\epsilon}$ is a constant depending on~$\epsilon$. \end{thm} \colb \par Using the ideas in the proof of Theorem~\ref{T03}, the same long time behavior can be obtained with initial data $(u_0 , \rho_0) \in D(A) \times H^2(\Omega)$, now without the compatibility condition \eqref{EWRTDFBSERSDFGSDFHDSFADSFSDFSVBASWERTDFSGSDFGFDGASFASDFDSFHFDGHFH12}. \par In the next theorem, we obtain the interior bounds for the $L^p$ norm of the Hessian $D^2 u$  of the velocity in the interior, for any $p\geq2$.  \par \cole \begin{thm} \label{T04} Let $(u_0 , \rho_0) \in (H^2(\Omega)\cap V) \times H^1(\Omega)$ and $p\in [2,\infty)$, and suppose that $\Omega'\subseteq \Omega$  is open and relatively compact. Then for the corresponding solution $(u,\rho)$ of \eqref{EWRTDFBSERSDFGSDFHDSFADSFSDFSVBASWERTDFSGSDFGFDGASFASDFDSFHFDGHFH01}--\eqref{EWRTDFBSERSDFGSDFHDSFADSFSDFSVBASWERTDFSGSDFGFDGASFASDFDSFHFDGHFH02}  and all $t_0>0$ we have a space-time bound   \begin{align}    \begin{split}     \Vert D^2 u \Vert_{L^p([t_0,T] ; L^p(\Omega'))} \leq C (T^{1/p} + 1)     ,    \end{split}    \label{EWRTDFBSERSDFGSDFHDSFADSFSDFSVBASWERTDFSGSDFGFDGASFASDFDSFHFDGHFH15}   \end{align} for $T\geq t_0>0$, while in addition we have a pointwise in time bound   \begin{align}    \begin{split}     \Vert D^2 u(t) \Vert_{L^p(\Omega')} \leq C t^{(p+4)/4}    \comma t\geq t_0    ,    \end{split}    \label{EWRTDFBSERSDFGSDFHDSFADSFSDFSVBASWERTDFSGSDFGFDGASFASDFDSFHFDGHFH16}   \end{align} where the constants in \eqref{EWRTDFBSERSDFGSDFHDSFADSFSDFSVBASWERTDFSGSDFGFDGASFASDFDSFHFDGHFH15} and \eqref{EWRTDFBSERSDFGSDFHDSFADSFSDFSVBASWERTDFSGSDFGFDGASFASDFDSFHFDGHFH16} depend on $t_0$, $p$, and $\dist(\Omega',\FGSDFHGFHDFGHDFGH\Omega)$. \end{thm} \colb \par \startnewsection{Proofs for the global bounds}{sec03} First, we recall prior results on the $L^2$ norms corresponding to Theorem~\ref{T01}. Let $(u_0,\rho_0) \in (H^2(\Omega)\cap V) \times H^1(\Omega)$. Then  there exists a unique  global solution  $(u,\rho)$ such that $u \in L^{\infty}((0,\infty),H^2(\Omega)) \cap L^2_{\loc}((0,\infty),H^3(\Omega))$ and $\rho \in L^{\infty}((0,\infty),H^1(\Omega))$ of \eqref{EWRTDFBSERSDFGSDFHDSFADSFSDFSVBASWERTDFSGSDFGFDGASFASDFDSFHFDGHFH01}--\eqref{EWRTDFBSERSDFGSDFHDSFADSFSDFSVBASWERTDFSGSDFGFDGASFASDFDSFHFDGHFH02}. Furthermore, the solution $(u,\rho)$ satisfies   \begin{align}   \begin{split}   &\Vert u(t) \Vert_{L^2}    +   \Vert \rho(t) \Vert_{L^2}   \les 1    \comma t\geq0   .   \end{split}    \label{EWRTDFBSERSDFGSDFHDSFADSFSDFSVBASWERTDFSGSDFGFDGASFASDFDSFHFDGHFH17}   \end{align} Here and below, the notation $a\les b$ means $a\leq C b$, where $C$ is a constant, which  is allowed to depend on the size of the initial data in the pertinent norms. We denote by   \begin{align}   B(u,v) = \mathbb{P}(u \cdot \grad v) \hspace{10mm} u,v \in V    \llabel{izY DvH TR zd9x SRVg 0Pl6Z8 9X z fLh GlH IYB x9 OELo 5loZ x4wag4 cn F aCE KfA 0uz fw HMUV M9Qy eARFe3 Py 6 kQG GFx rPf 6T ZBQR la1a 6Aeker Xg k blz nSm mhY jc z3io WYjz h33sxR JM k Dos EAA hUO Oz aQfK Z0cn 5kqYPn W7 1 vCT 69a EC9 LD EQ5S BK4J fVFLAo Qp N dzZ HAl JaL Mn vRqH 7pBB qOr7fv oa e BSA 8TE btx y3 jwK3 v244 dlfwRL Dc g X14 vTp Wd8 zyEWRTDFBSERSDFGSDFHDSFADSFSDFSVBASWERTDFSGSDFGFDGASFASDFDSFHFDGHFH18}   \end{align} the bilinear term corresponding to the Navier-Stokes equations. This allows us to rewrite \eqref{EWRTDFBSERSDFGSDFHDSFADSFSDFSVBASWERTDFSGSDFGFDGASFASDFDSFHFDGHFH01} as   \begin{align}   \begin{split}   &u_t + Au + B(u,u)  = \mathbb{P}(\rho e_2) \\&   \rho_t + u \cdot \grad \rho = 0   .   \end{split}   \label{EWRTDFBSERSDFGSDFHDSFADSFSDFSVBASWERTDFSGSDFGFDGASFASDFDSFHFDGHFH19}    \end{align} \colb \par We now turn to the proof of the first theorem. \par \begin{proof}[Proof of Theorem~\ref{T01}] We begin by proving that $\Vert u\Vert_{L^2}$ dissipates. 
Inspired by \cite{DWZZ}, we shift the  density by $x_2$, i.e., introduce    \begin{equation}    \theta(x_1,x_2,t)= \rho(x_1,x_2,t) - x_2    ,    \label{EWRTDFBSERSDFGSDFHDSFADSFSDFSVBASWERTDFSGSDFGFDGASFASDFDSFHFDGHFH123}   \end{equation} and compensate with $P=p(x_1,x_2,t)-x_2^2/2$ to derive an equivalent system of equations   \begin{align}   \begin{split}   & u_t - \Laplace u + u \cdot \grad u + \grad P = \theta e_2    \\&   \theta_t + u \cdot \grad \theta = -u \cdot e_2    \\&   \grad \cdot u = 0   ,   \end{split}   \label{EWRTDFBSERSDFGSDFHDSFADSFSDFSVBASWERTDFSGSDFGFDGASFASDFDSFHFDGHFH20}   \end{align} with $u\bigl|_{\FGSDFHGFHDFGHDFGH \Omega} = 0$.  Multiplying the first equation of \eqref{EWRTDFBSERSDFGSDFHDSFADSFSDFSVBASWERTDFSGSDFGFDGASFASDFDSFHFDGHFH20} with $u$ and the second by $\theta$, integrating, and applying the Dirichlet boundary conditions and incompressibility, we obtain   \begin{align}   \begin{split}   \frac{1}{2}\frac{d}{dt}(\Vert u \Vert_{L^2}^2 + \Vert \theta \Vert_{L^2}^2) + \Vert \grad u \Vert_{L^2}^2 = 0   .   \end{split}   \label{EWRTDFBSERSDFGSDFHDSFADSFSDFSVBASWERTDFSGSDFGFDGASFASDFDSFHFDGHFH21}   \end{align} Observe that the norm $\Vert \theta\Vert_{L^2}$ may increase, thus no direct conclusion on decay rates can be reached from~\eqref{EWRTDFBSERSDFGSDFHDSFADSFSDFSVBASWERTDFSGSDFGFDGASFASDFDSFHFDGHFH21}. The identity \eqref{EWRTDFBSERSDFGSDFHDSFADSFSDFSVBASWERTDFSGSDFGFDGASFASDFDSFHFDGHFH21} implies $\Vert u \Vert_{L^2}^2$ and $\Vert \theta \Vert_{L^2}^2$ are uniformly bounded in time  and    \begin{equation}    \int_{0}^{\infty}\Vert \grad u \Vert_{L^2}^2 \les 1    ,    \llabel{ YWjw eQmF yD5y5l DN l ZbA Jac cld kx Yn3V QYIV v6fwmH z1 9 w3y D4Y ezR M9 BduE L7D9 2wTHHc Do g ZxZ WRW Jxi pv fz48 ZVB7 FZtgK0 Y1 w oCo hLA i70 NO Ta06 u2sY GlmspV l2 x y0X B37 x43 k5 kaoZ deyE sDglRF Xi 9 6b6 w9B dId Ko gSUM NLLb CRzeQL UZ m i9O 2qv VzD hz v1r6 spSl jwNhG6 s6 i SdX hob hbp 2u sEdl 95LP AtrBBi bP C wSh pFC CUa yz xYS5 78roEWRTDFBSERSDFGSDFHDSFADSFSDFSVBASWERTDFSGSDFGFDGASFASDFDSFHFDGHFH22}   \end{equation} where we allow all constants to depend on $\Vert u_0\Vert_{H^2}$ and $\Vert \rho_0\Vert_{H^{1}}$. Utilizing the Poincar\'e inequality, we also get   \begin{equation}    \int_{0}^{\infty}\Vert  u \Vert_{L^2}^2     \les 1    .    \label{EWRTDFBSERSDFGSDFHDSFADSFSDFSVBASWERTDFSGSDFGFDGASFASDFDSFHFDGHFH23}   \end{equation} 
To prove the uniform continuity from above of the $L^2$ norm of $u$,  we multiply the first equation in \eqref{EWRTDFBSERSDFGSDFHDSFADSFSDFSVBASWERTDFSGSDFGFDGASFASDFDSFHFDGHFH20} with $u$ and integrate by parts to find that   \begin{align}   \begin{split}   \frac{1}{2}\frac{d}{dt} \Vert u \Vert_{L^2}^2 + \Vert \grad u \Vert_{L^2}^2    = \int_{\Omega} \theta u \cdot e_2   \leq \Vert u \Vert_{L^2} \Vert \theta \Vert_{L^2}   \les \Vert u \Vert_{L^2}   ,   \end{split}   \llabel{ f3UwDP sC I pES HB1 qFP SW 5tt0 I7oz jXun6c z4 c QLB J4M NmI 6F 08S2 Il8C 0JQYiU lI 1 YkK oiu bVt fG uOeg Sllv b4HGn3 bS Z LlX efa eN6 v1 B6m3 Ek3J SXUIjX 8P d NKI UFN JvP Ha Vr4T eARP dXEV7B xM 0 A7w 7je p8M 4Q ahOi hEVo Pxbi1V uG e tOt HbP tsO 5r 363R ez9n A5EJ55 pc L lQQ Hg6 X1J EW K8Cf 9kZm 14A5li rN 7 kKZ rY0 K10 It eJd3 kMGw opVnfY EGEWRTDFBSERSDFGSDFHDSFADSFSDFSVBASWERTDFSGSDFGFDGASFASDFDSFHFDGHFH24}   \end{align} which, by Poincar\'e and Young's inequalities, implies   \begin{align}   \begin{split}   \frac{d}{dt} \Vert u \Vert_{L^2}^2 + \Vert \grad u \Vert_{L^2}^2    \les 1   .   \end{split}   \label{EWRTDFBSERSDFGSDFHDSFADSFSDFSVBASWERTDFSGSDFGFDGASFASDFDSFHFDGHFH139}   \end{align} It is elementary to show that if a differentiable function $f\colon [0,\infty)\to[0,\infty)$ satisfies $\int_{0}^{\infty} f(s)\,ds<\infty$ and $f'(t)\les 1$, then $\lim_{t\to\infty}f(t)=0$. Applying the statement with $f(t)=\Vert u\Vert_{L^2}^2$, the inequalities \eqref{EWRTDFBSERSDFGSDFHDSFADSFSDFSVBASWERTDFSGSDFGFDGASFASDFDSFHFDGHFH23} and \eqref{EWRTDFBSERSDFGSDFHDSFADSFSDFSVBASWERTDFSGSDFGFDGASFASDFDSFHFDGHFH139} imply   \begin{align}   \Vert u \Vert_{L^2} \to 0   \as{$t\to\infty$}   .   \label{EWRTDFBSERSDFGSDFHDSFADSFSDFSVBASWERTDFSGSDFGFDGASFASDFDSFHFDGHFH26}   \end{align} Next, we  aim to prove that  $\Vert \grad u \Vert_{L^2}^2 \to 0$. We take the $L^2$ inner product of  \eqref{EWRTDFBSERSDFGSDFHDSFADSFSDFSVBASWERTDFSGSDFGFDGASFASDFDSFHFDGHFH19}$_1$  with $Au$ to find that   \begin{align}   \begin{split}   &   \frac{1}{2}\frac{d}{dt} \Vert A^{1/2}u\Vert_{L^2}^2 + \Vert Au \Vert_{L^2}^2    = -\langle  B(u,u),Au\rangle_{L^2} + \langle \mathbb{P}(\theta e_2), Au\rangle_{L^2}   \\&\indeq   \leq    \Vert B(u,u)\Vert_{L^2}\Vert Au \Vert_{L^2} + \Vert \theta \Vert_{L^2}\Vert Au\Vert_{L^2}   \les   \Vert u\Vert_{L^2}^{1/2}   \Vert A^{1/2}u\Vert_{L^2}   \Vert Au\Vert_{L^2}^{3/2}   + \Vert Au\Vert_{L^2}   ,   \end{split}   \label{EWRTDFBSERSDFGSDFHDSFADSFSDFSVBASWERTDFSGSDFGFDGASFASDFDSFHFDGHFH28}   \end{align} where we used   \begin{equation}    \Vert B(u,u)\Vert_{L^2}    \les    \Vert u\Vert_{L^4} \Vert \nabla u\Vert_{L^4}    \les    \Vert u\Vert_{L^2}^{1/2}    \Vert u\Vert_{H^{1}}    \Vert u\Vert_{H^2}^{1/2}    \les    \Vert u\Vert_{L^2}^{1/2}    \Vert A^{1/2} u\Vert_{L^2}    \Vert A u\Vert_{L^2}^{1/2}    .    \label{EWRTDFBSERSDFGSDFHDSFADSFSDFSVBASWERTDFSGSDFGFDGASFASDFDSFHFDGHFH131}   \end{equation} In \eqref{EWRTDFBSERSDFGSDFHDSFADSFSDFSVBASWERTDFSGSDFGFDGASFASDFDSFHFDGHFH28}, we apply Young's inequality and absorb the factors $\Vert A u\Vert_{L^2}$ into the second term on the left side, obtaining   \begin{align}   \begin{split}   &   \frac{d}{dt} \Vert A^{1/2}u\Vert_{L^2}^2 + \Vert Au \Vert_{L^2}^2    \les   \Vert u\Vert_{L^2}^{2}   \Vert A^{1/2}u\Vert_{L^2}^{4}   + 1   \les   \Vert A^{1/2}u\Vert_{L^2}^{4}   + 1   .   \end{split}   \llabel{ 2 orG fj0 TTA Xt ecJK eTM0 x1N9f0 lR p QkP M37 3r0 iA 6EFs 1F6f 4mjOB5 zu 5 GGT Ncl Bmk b5 jOOK 4yny My04oz 6m 6 Akz NnP JXh Bn PHRu N5Ly qSguz5 Nn W 2lU Yx3 fX4 hu LieH L30w g93Xwc gj 1 I9d O9b EPC R0 vc6A 005Q VFy1ly K7 o VRV pbJ zZn xY dcld XgQa DXY3gz x3 6 8OR JFK 9Uh XT e3xY bVHG oYqdHg Vy f 5kK Qzm mK4 9x xiAp jVkw gzJOdE 4v g hAv 9bVEWRTDFBSERSDFGSDFHDSFADSFSDFSVBASWERTDFSGSDFGFDGASFASDFDSFHFDGHFH29}   \end{align} Utilizing  Lemma~\ref{L01} in the Appendix, we obtain   \begin{equation}    \Vert A^{1/2}u(t)\Vert_{L^2}    \les 1    \comma t\geq0          \label{EWRTDFBSERSDFGSDFHDSFADSFSDFSVBASWERTDFSGSDFGFDGASFASDFDSFHFDGHFH122}   \end{equation} and    \begin{equation}    \Vert A^{1/2}u(t)\Vert_{L^2}    \to 0    \as{$t\to\infty$}    ,    \llabel{ IHe wc Vqcb SUcF 1pHzol Nj T l1B urc Sam IP zkUS 8wwS a7wVWR 4D L VGf 1RF r59 9H tyGq hDT0 TDlooa mg j 9am png aWe nG XU2T zXLh IYOW5v 2d A rCG sLk s53 pW AuAy DQlF 6spKyd HT 9 Z1X n2s U1g 0D Llao YuLP PB6YKo D1 M 0fi qHU l4A Ia joiV Q6af VT6wvY Md 0 pCY BZp 7RX Hd xTb0 sjJ0 Beqpkc 8b N OgZ 0Tr 0wq h1 C2Hn YQXM 8nJ0Pf uG J Be2 vuq Duk LV AJEWRTDFBSERSDFGSDFHDSFADSFSDFSVBASWERTDFSGSDFGFDGASFASDFDSFHFDGHFH138}   \end{equation} giving \eqref{EWRTDFBSERSDFGSDFHDSFADSFSDFSVBASWERTDFSGSDFGFDGASFASDFDSFHFDGHFH09}. In addition, by the same lemma,   \begin{equation}    \limsup_{t\to\infty}     \int_{t}^{t+t_0} \Vert A u\Vert_{L^2}^2    \les t_0    \comma t_0\geq 0    .    \label{EWRTDFBSERSDFGSDFHDSFADSFSDFSVBASWERTDFSGSDFGFDGASFASDFDSFHFDGHFH30}   \end{equation} We note in passing, and since it is needed in the proof of Theorem~\ref{T04}, that the inequality of type \eqref{EWRTDFBSERSDFGSDFHDSFADSFSDFSVBASWERTDFSGSDFGFDGASFASDFDSFHFDGHFH30} also holds with $A u$ replaced with~$u_t$. To show that $u_t$ dissipates in the $L^2$ norm, we take the time derivative of \eqref{EWRTDFBSERSDFGSDFHDSFADSFSDFSVBASWERTDFSGSDFGFDGASFASDFDSFHFDGHFH20}$_1$,  multiply by $u_t$, and integrate by parts, to get the equation   \begin{align}   \begin{split}   &\frac{1}{2} \frac{d}{dt} \Vert u_t \Vert_{L^2}^2      + \Vert \grad u_t \Vert_{L^2}^2      = \langle \theta_t e_2, u_t\rangle_{L^2}        - \langle u_t \cdot \grad u, u_t\rangle_{L^2}   .   \end{split}   \label{EWRTDFBSERSDFGSDFHDSFADSFSDFSVBASWERTDFSGSDFGFDGASFASDFDSFHFDGHFH33}   \end{align} For the first term on the right, we apply \eqref{EWRTDFBSERSDFGSDFHDSFADSFSDFSVBASWERTDFSGSDFGFDGASFASDFDSFHFDGHFH20}$_2$ to obtain   \begin{align}   \begin{split}    \langle \theta_t e_2, u_t\rangle_{L^2}       &= -\int_{\Omega} (u \cdot \grad \theta)(\FGSDFHGFHDFGHDFGH_{t}u_2)          - \int_{\Omega} u_2 \FGSDFHGFHDFGHDFGH_{t} u_2      = \int_{\Omega} \theta u \cdot \grad \FGSDFHGFHDFGHDFGH_{t}u_2        - \int_{\Omega} u_2 \FGSDFHGFHDFGHDFGH_{t} u_2    \\&    \les    \Vert \theta \Vert_{L^4}     \Vert u\Vert_{L^{2}}^{1/2}    \Vert A^{1/2}u\Vert_{L^{2}}^{1/2}     \Vert \grad u_t \Vert_{L^2}       + \Vert u \Vert_{L^2} \Vert u_t \Vert_{L^2}    \\&    \les    \Vert A^{1/2}u\Vert_{L^{2}}^{1/2}     \Vert \grad u_t \Vert_{L^2}       +      \Vert u\Vert_{L^{2}}     \Vert \nabla u_t \Vert_{L^2}    ,    \end{split}    \label{EWRTDFBSERSDFGSDFHDSFADSFSDFSVBASWERTDFSGSDFGFDGASFASDFDSFHFDGHFH34}   \end{align} where we used $\Vert \theta\Vert_{L^4}\les 1$  and $\Vert u_t\Vert_{L^2}\les \Vert \nabla u_t\Vert_{L^2}$ in the last inequality. For the second term on the right-hand side  of \eqref{EWRTDFBSERSDFGSDFHDSFADSFSDFSVBASWERTDFSGSDFGFDGASFASDFDSFHFDGHFH33}, we write   \begin{equation}    - \langle u_t \cdot \grad u, u_t\rangle_{L^2}    \les    \Vert u_t\Vert_{L^4}^2    \Vert \nabla u\Vert_{L^2}    \les    \Vert u_t\Vert_{L^2}    \Vert \nabla u_t\Vert_{L^2}     \Vert A^{1/2}u\Vert_{L^2}    .    \label{EWRTDFBSERSDFGSDFHDSFADSFSDFSVBASWERTDFSGSDFGFDGASFASDFDSFHFDGHFH35}   \end{equation} Using \eqref{EWRTDFBSERSDFGSDFHDSFADSFSDFSVBASWERTDFSGSDFGFDGASFASDFDSFHFDGHFH34} and \eqref{EWRTDFBSERSDFGSDFHDSFADSFSDFSVBASWERTDFSGSDFGFDGASFASDFDSFHFDGHFH35} in \eqref{EWRTDFBSERSDFGSDFHDSFADSFSDFSVBASWERTDFSGSDFGFDGASFASDFDSFHFDGHFH33} and then absorbing the factors $\Vert \nabla u_t\Vert_{L^2}$ by Young's inequality, we get   \begin{align}   \begin{split}   & \frac{d}{dt} \Vert u_t \Vert_{L^2}^2      + \Vert \grad u_t \Vert_{L^2}^2      \les     \Vert A^{1/2}u\Vert_{L^{2}}     + \Vert u \Vert_{L^2}^2     + \Vert u_t\Vert_{L^2}^2       \Vert A^{1/2}u\Vert_{L^2}^2    \les    \phi(t) (1+\Vert u_t\Vert_{L^2}^2)   ,   \end{split}    \llabel{wv 2tYc JOM1uK h7 p cgo iiK t0b 3e URec DVM7 ivRMh1 T6 p AWl upj kEj UL R3xN VAu5 kEbnrV HE 1 OrJ 2bx dUP yD vyVi x6sC BpGDSx jB C n9P Fiu xkF vw 0QPo fRjy 2OFItV eD B tDz lc9 xVy A0 de9Y 5h8c 7dYCFk Fl v WPD SuN VI6 MZ 72u9 MBtK 9BGLNs Yp l X2y b5U HgH AD bW8X Rzkv UJZShW QH G oKX yVA rsH TQ 1Vbd dK2M IxmTf6 wE T 9cX Fbu uVx Cb SBBp 0v2J MQEWRTDFBSERSDFGSDFHDSFADSFSDFSVBASWERTDFSGSDFGFDGASFASDFDSFHFDGHFH36}   \end{align} where $\phi\colon[0,\infty)\to[0,\infty)$ is a bounded function, which satisfies $\lim_{t\to\infty}\phi(t)=0$. By Lemma~\ref{L02}, we get   \begin{equation}    \Vert u_t\Vert_{L^2}\les 1    \comma t\in[0,\infty)       \label{EWRTDFBSERSDFGSDFHDSFADSFSDFSVBASWERTDFSGSDFGFDGASFASDFDSFHFDGHFH163}   \end{equation} and   \begin{align}    \Vert u_t(t) \Vert_{L^2} \to 0    \as{$t\to \infty$}    \label{EWRTDFBSERSDFGSDFHDSFADSFSDFSVBASWERTDFSGSDFGFDGASFASDFDSFHFDGHFH37}   \end{align} as well as   \begin{equation}    \limsup_{t\to\infty}     \int_{t}^{t+t_0} \Vert \nabla u_t\Vert_{L^2}^2    = 0    \comma t_0\geq 0    .    \label{EWRTDFBSERSDFGSDFHDSFADSFSDFSVBASWERTDFSGSDFGFDGASFASDFDSFHFDGHFH38}   \end{equation} \par Next, from \eqref{EWRTDFBSERSDFGSDFHDSFADSFSDFSVBASWERTDFSGSDFGFDGASFASDFDSFHFDGHFH19}$_1$, we obtain   \begin{align}    \begin{split}    \Vert Au\Vert_{L^2}    &\les    \Vert u_t\Vert_{L^2}    + \Vert B(u,u)\Vert_{L^2}    + \Vert \rho\Vert_{L^2}    \les    \Vert u_t\Vert_{L^2}    + \Vert u\Vert_{L^2}^{1/2}      \Vert A^{1/2}u\Vert_{L^2}      \Vert A u\Vert_{L^2}^{1/2}    +  1          .    \end{split}    \llabel{5Z8z 3p M EGp TU6 KCc YN 2BlW dp2t mliPDH JQ W jIR Rgq i5l AP gikl c8ru HnvYFM AI r Ih7 Ths 9tE hA AYgS swZZ fws19P 5w e JvM imb sFH Th CnSZ HORm yt98w3 U3 z ant zAy Twq 0C jgDI Etkb h98V4u o5 2 jjA Zz1 kLo C8 oHGv Z5Ru Gwv3kK 4W B 50T oMt q7Q WG 9mtb SIlc 87ruZf Kw Z Ph3 1ZA Osq 8l jVQJ LTXC gyQn0v KE S iSq Bpa wtH xc IJe4 SiE1 izzxim ke P EWRTDFBSERSDFGSDFHDSFADSFSDFSVBASWERTDFSGSDFGFDGASFASDFDSFHFDGHFH39}   \end{align} Absorbing the factor $\Vert Au\Vert_{L^2}^{1/2}$ in the left-hand side by using Young's inequality, we get   \begin{equation}    \Vert Au\Vert_{L^2}    \les     \Vert u_t\Vert_{L^2}    + \Vert u\Vert_{L^2}      \Vert A^{1/2}u\Vert_{L^2}^2    +  1          ,    \llabel{Y3s 7SX 5DA SG XHqC r38V YP3Hxv OI R ZtM fqN oLF oU 7vNd txzw UkX32t 94 n Fdq qTR QOv Yq Ebig jrSZ kTN7Xw tP F gNs O7M 1mb DA btVB 3LGC pgE9hV FK Y LcS GmF 863 7a ZDiz 4CuJ bLnpE7 yl 8 5jg Many Thanks, POL OG EPOe Mru1 v25XLJ Fz h wgE lnu Ymq rX 1YKV Kvgm MK7gI4 6h 5 kZB OoJ tfC 5g VvA1 kNJr 2o7om1 XN p Uwt CWX fFT SW DjsI wuxO JxLU1S xA 5 OEWRTDFBSERSDFGSDFHDSFADSFSDFSVBASWERTDFSGSDFGFDGASFASDFDSFHFDGHFH40}   \end{equation} from where, by \eqref{EWRTDFBSERSDFGSDFHDSFADSFSDFSVBASWERTDFSGSDFGFDGASFASDFDSFHFDGHFH09} and  \eqref{EWRTDFBSERSDFGSDFHDSFADSFSDFSVBASWERTDFSGSDFGFDGASFASDFDSFHFDGHFH37}, we get \eqref{EWRTDFBSERSDFGSDFHDSFADSFSDFSVBASWERTDFSGSDFGFDGASFASDFDSFHFDGHFH08}. Note, in passing, that \eqref{EWRTDFBSERSDFGSDFHDSFADSFSDFSVBASWERTDFSGSDFGFDGASFASDFDSFHFDGHFH08} and \eqref{EWRTDFBSERSDFGSDFHDSFADSFSDFSVBASWERTDFSGSDFGFDGASFASDFDSFHFDGHFH26} imply   \begin{equation}       \Vert u(t) \Vert_{L^\infty} \to 0    \as{$t\to \infty$}     ,    \label{EWRTDFBSERSDFGSDFHDSFADSFSDFSVBASWERTDFSGSDFGFDGASFASDFDSFHFDGHFH41}   \end{equation} by Agmon's inequality. From \eqref{EWRTDFBSERSDFGSDFHDSFADSFSDFSVBASWERTDFSGSDFGFDGASFASDFDSFHFDGHFH19}$_1$, we get   \begin{align}    \begin{split}    \Vert Au - \mathbb{P}(\rho e_2)\Vert_{L^2}    & \les    \Vert u_t\Vert_{L^2}    + \Vert B(u,u)\Vert_{L^2}    \les    \Vert u_t\Vert_{L^2}    +    \Vert u\Vert_{L^2}^{1/2}    \Vert A^{1/2}u\Vert_{L^2}    \Vert A u\Vert_{L^2}^{1/2}    .    \end{split}    \label{EWRTDFBSERSDFGSDFHDSFADSFSDFSVBASWERTDFSGSDFGFDGASFASDFDSFHFDGHFH42}   \end{align} By \eqref{EWRTDFBSERSDFGSDFHDSFADSFSDFSVBASWERTDFSGSDFGFDGASFASDFDSFHFDGHFH08}, \eqref{EWRTDFBSERSDFGSDFHDSFADSFSDFSVBASWERTDFSGSDFGFDGASFASDFDSFHFDGHFH09}, \eqref{EWRTDFBSERSDFGSDFHDSFADSFSDFSVBASWERTDFSGSDFGFDGASFASDFDSFHFDGHFH26}, and \eqref{EWRTDFBSERSDFGSDFHDSFADSFSDFSVBASWERTDFSGSDFGFDGASFASDFDSFHFDGHFH37}, the right-hand side of \eqref{EWRTDFBSERSDFGSDFHDSFADSFSDFSVBASWERTDFSGSDFGFDGASFASDFDSFHFDGHFH42} converges to $0$ as $t\to\infty$, and we obtain \eqref{EWRTDFBSERSDFGSDFHDSFADSFSDFSVBASWERTDFSGSDFGFDGASFASDFDSFHFDGHFH10}. \par We lastly proceed to prove the $o(1)$-type exponential estimate on the growth of $\Vert \grad \theta \Vert_{L^2}$. For this, we first need to prove the local in time boundedness of $\Vert \theta\Vert_{H^{1}}$, which in turn requires us to first  bound $\int_{0}^{T}\Vert \nabla u\Vert_{L^\infty}$ for some $T>0$. As above, we have   \begin{equation}    \int_{0}^{T}      \Vert \nabla u_t\Vert_{L^2}^2     \les 1    ,    \llabel{bG 3IO UdL qJ cCAr gzKM 08DvX2 mu i 13T t71 Iwq oF UI0E Ef5S V2vxcy SY I QGr qrB HID TJ v1OB 1CzD IDdW4E 4j J mv6 Ktx oBO s9 ADWB q218 BJJzRy UQ i 2Gp weE T8L aO 4ho9 5g4v WQmoiq jS w MA9 Cvn Gqx l1 LrYu MjGb oUpuvY Q2 C dBl AB9 7ew jc 5RJE SFGs ORedoM 0b B k25 VEK B8V A9 ytAE Oyof G8QIj2 7a I 3jy Rmz yET Kx pgUq 4Bvb cD1b1g KB y oE3 azg elVEWRTDFBSERSDFGSDFHDSFADSFSDFSVBASWERTDFSGSDFGFDGASFASDFDSFHFDGHFH155}   \end{equation} for all $T>0$, where the constant depends on $T$. Now, consider the Stokes problem   \begin{align}    \begin{split}    &     u_t -\Delta u + \nabla p = - u\cdot\nabla u + \rho e_2     \\&     \nabla \cdot u=0     \\&     u|_{\FGSDFHGFHDFGHDFGH\Omega} = 0     .    \end{split}    \llabel{ Nu 8iZ1 w1tq twKx8C LN 2 8yn jdo jUW vN H9qy HaXZ GhjUgm uL I 87i Y7Q 9MQ Wa iFFS Gzt8 4mSQq2 5O N ltT gbl 8YD QS AzXq pJEK 7bGL1U Jn 0 f59 vPr wdt d6 sDLj Loo1 8tQXf5 5u p mTa dJD sEL pH 2vqY uTAm YzDg95 1P K FP6 pEi zIJ Qd 8Ngn HTND 6z6ExR XV 0 ouU jWT kAK AB eAC9 Rfja c43Ajk Xn H dgS y3v 5cB et s3VX qfpP BqiGf9 0a w g4d W9U kvR iJ y46G bEWRTDFBSERSDFGSDFHDSFADSFSDFSVBASWERTDFSGSDFGFDGASFASDFDSFHFDGHFH25}   \end{align} By \cite[Theorem~2.7]{SvW}  (see also \cite{GS})  applied with $s=p=3$, we obtain that for any $\tepsilon>0$   \begin{align}    \begin{split}    \int_{0}^{T}      \Vert u\Vert_{W^{2,3}}^{3}    \les    \Vert A_3^{2/3+\tepsilon}u_0\Vert_{L^3}^{3}    + \int_{0}^{T}    \Vert  u\cdot\nabla u - \rho e_2 \Vert_{L^3}^{3}    ,    \end{split}    \label{EWRTDFBSERSDFGSDFHDSFADSFSDFSVBASWERTDFSGSDFGFDGASFASDFDSFHFDGHFH167}   \end{align} for all $T>0$, where the constant depends on $T$ and $\tepsilon$. In \eqref{EWRTDFBSERSDFGSDFHDSFADSFSDFSVBASWERTDFSGSDFGFDGASFASDFDSFHFDGHFH167}, $A_3$ denotes the $L^3$ version of the Stokes operator (cf.~\cite{SvW}). For the first term on the right-hand side in \eqref{EWRTDFBSERSDFGSDFHDSFADSFSDFSVBASWERTDFSGSDFGFDGASFASDFDSFHFDGHFH167}, we use   \begin{equation}    \Vert A_3^{2/3+\tepsilon}u_0\Vert_{L^3}    \les \Vert A u_0\Vert_{L^2}    \les 1    \label{EWRTDFBSERSDFGSDFHDSFADSFSDFSVBASWERTDFSGSDFGFDGASFASDFDSFHFDGHFH169}   \end{equation} with $\tepsilon=1/6$  from the embedding property on \cite[p.~430]{SvW}, while for the second term we estimate   \begin{align}    \begin{split}    &\Vert  u\cdot\nabla u - \rho e_2 \Vert_{L^3}^{3}    \les    \Vert u\Vert_{L^6}^{3}    \Vert \nabla u\Vert_{L^6}^{3}    + \Vert \rho\Vert_{L^3}^{3}    \les    \Vert u\Vert_{L^2}    \Vert A^{1/2}u\Vert_{L^2}^{3}    \Vert Au\Vert_{L^2}^{2}    +    1    \les 1    .    \end{split}    \label{EWRTDFBSERSDFGSDFHDSFADSFSDFSVBASWERTDFSGSDFGFDGASFASDFDSFHFDGHFH168}   \end{align} Applying \eqref{EWRTDFBSERSDFGSDFHDSFADSFSDFSVBASWERTDFSGSDFGFDGASFASDFDSFHFDGHFH169} and \eqref{EWRTDFBSERSDFGSDFHDSFADSFSDFSVBASWERTDFSGSDFGFDGASFASDFDSFHFDGHFH168} in \eqref{EWRTDFBSERSDFGSDFHDSFADSFSDFSVBASWERTDFSGSDFGFDGASFASDFDSFHFDGHFH167}, we get   \begin{equation}    \int_{0}^{T}     \Vert D^2 u\Vert_{L^3}^{3}     \les 1     ,    \label{EWRTDFBSERSDFGSDFHDSFADSFSDFSVBASWERTDFSGSDFGFDGASFASDFDSFHFDGHFH170}   \end{equation} where the constant depends on~$T$ and consequently   \begin{align}    \begin{split}    \int_{0}^{T}    \Vert \nabla u\Vert_{L^\infty}    \les 1    \end{split}    \label{EWRTDFBSERSDFGSDFHDSFADSFSDFSVBASWERTDFSGSDFGFDGASFASDFDSFHFDGHFH159}   \end{align} for all $T>0$, where the constant depends on $T$, due to the Gagliardo-Nirenberg type inequality   \begin{equation}    \Vert v\Vert_{L^\infty}    \les    \Vert v\Vert_{L^2}^{1/4}    \Vert \nabla v\Vert_{L^3}^{3/4}    + \Vert v\Vert_{L^2}    .    \label{EWRTDFBSERSDFGSDFHDSFADSFSDFSVBASWERTDFSGSDFGFDGASFASDFDSFHFDGHFH32}   \end{equation} By applying the gradient to \eqref{EWRTDFBSERSDFGSDFHDSFADSFSDFSVBASWERTDFSGSDFGFDGASFASDFDSFHFDGHFH20}$_2$ and taking the inner product with $\grad \theta$, we find that   \begin{align}   \begin{split}   \frac{1}{2}\frac{d}{dt} \Vert \grad \theta \Vert_{L^2}^2   = - \langle \grad(u \cdot \grad \theta), \grad \theta\rangle_{L^2} - \langle  \grad (u \cdot e_2), \grad \theta\rangle_{L^2}   .   \end{split}   \label{EWRTDFBSERSDFGSDFHDSFADSFSDFSVBASWERTDFSGSDFGFDGASFASDFDSFHFDGHFH44}   \end{align} The second term is estimated  by $C\Vert \grad u \Vert_{L^2} \Vert \grad \theta \Vert_{L^2}$,  using the Cauchy-Schwarz inequality.  The first term is likewise bounded as   \begin{align}   \begin{split}   -   \langle \grad(u \cdot \grad \theta), \grad \theta\rangle_{L^2}    &    -= \int_{\Omega} \FGSDFHGFHDFGHDFGH_j (u_i \FGSDFHGFHDFGHDFGH_i \theta) \FGSDFHGFHDFGHDFGH_j \theta   =     -\int_{\Omega} \FGSDFHGFHDFGHDFGH_j u_i \FGSDFHGFHDFGHDFGH_i \theta \FGSDFHGFHDFGHDFGH_j \theta      -\frac{1}{2}   \int_{\Omega} u_i \FGSDFHGFHDFGHDFGH_i |\nabla \theta|^2   \les      \Vert \grad u \Vert_{L^\infty} \Vert \grad \theta \Vert_{L^2}^2   ,   \end{split}    \llabel{H3U cJ86hW Va C Mje dsU cqD SZ 1DlP 2mfB hzu5dv u1 i 6eW 2YN LhM 3f WOdz KS6Q ov14wx YY d 8sa S38 hIl cP tS4l 9B7h FC3JXJ Gp s tll 7a7 WNr VM wunm nmDc 5duVpZ xT C l8F I01 jhn 5B l4Jz aEV7 CKMThL ji 1 gyZ uXc Iv4 03 3NqZ LITG Ux3ClP CB K O3v RUi mJq l5 blI9 GrWy irWHof lH 7 3ZT eZX kop eq 8XL1 RQ3a Uj6Ess nj 2 0MA 3As rSV ft 3F9w zB1q DQVOnHEWRTDFBSERSDFGSDFHDSFADSFSDFSVBASWERTDFSGSDFGFDGASFASDFDSFHFDGHFH45}   \end{align} by \eqref{EWRTDFBSERSDFGSDFHDSFADSFSDFSVBASWERTDFSGSDFGFDGASFASDFDSFHFDGHFH20}$_3$ and  $u\bigl|_{\FGSDFHGFHDFGHDFGH \Omega} = 0$. Thus, estimating the two terms in \eqref{EWRTDFBSERSDFGSDFHDSFADSFSDFSVBASWERTDFSGSDFGFDGASFASDFDSFHFDGHFH44} as indicated, 
we conclude that    \begin{align}   \begin{split}   \frac{d}{dt} \Vert \grad \theta \Vert_{L^2}    \les    \Vert \grad u \Vert_{L^\infty} \Vert \grad \theta \Vert_{L^2}   + \Vert \nabla u\Vert_{L^2}   \les    \Vert \grad u \Vert_{L^\infty} (\Vert \grad \theta \Vert_{L^2}+1)   ,   \end{split}   \label{EWRTDFBSERSDFGSDFHDSFADSFSDFSVBASWERTDFSGSDFGFDGASFASDFDSFHFDGHFH46}   \end{align} which implies that the exponential growth of $\Vert \grad \theta \Vert_{L^2}$ is determined by  the time integral of $\Vert \grad u \Vert_{L^\infty}$.  In particular, applying \eqref{EWRTDFBSERSDFGSDFHDSFADSFSDFSVBASWERTDFSGSDFGFDGASFASDFDSFHFDGHFH159} to \eqref{EWRTDFBSERSDFGSDFHDSFADSFSDFSVBASWERTDFSGSDFGFDGASFASDFDSFHFDGHFH46} yields   \begin{equation}    \Vert \theta\Vert_{H^{1}}    \les 1    \comma t\in[0,T]    ,    \label{EWRTDFBSERSDFGSDFHDSFADSFSDFSVBASWERTDFSGSDFGFDGASFASDFDSFHFDGHFH160}   \end{equation} for all $T>0$, where the constant depends on $T$. \par Next, we  fix $\epsilon\in(0,1]$ and claim that    \begin{align}   \begin{split}   \Vert \theta(t) \Vert_{H^1}    \les e^{\epsilon t}    \comma t\geq 0   ,   \end{split}    \label{EWRTDFBSERSDFGSDFHDSFADSFSDFSVBASWERTDFSGSDFGFDGASFASDFDSFHFDGHFH43}   \end{align} where we allow all constants to depend on $\epsilon$. Note that \eqref{EWRTDFBSERSDFGSDFHDSFADSFSDFSVBASWERTDFSGSDFGFDGASFASDFDSFHFDGHFH43} directly implies \eqref{EWRTDFBSERSDFGSDFHDSFADSFSDFSVBASWERTDFSGSDFGFDGASFASDFDSFHFDGHFH11} by the definition \eqref{EWRTDFBSERSDFGSDFHDSFADSFSDFSVBASWERTDFSGSDFGFDGASFASDFDSFHFDGHFH123}. To prove \eqref{EWRTDFBSERSDFGSDFHDSFADSFSDFSVBASWERTDFSGSDFGFDGASFASDFDSFHFDGHFH43}, we need to estimate the time integral of $\Vert \nabla u\Vert_{L^\infty}$. Let $0<t_0\leq t_1$, where $t_0\geq2$ is a large time to be determined based on $\epsilon$. By the Gagliardo-Nirenberg in space and H\"older's inequalities in time, we have, using \eqref{EWRTDFBSERSDFGSDFHDSFADSFSDFSVBASWERTDFSGSDFGFDGASFASDFDSFHFDGHFH32}   \begin{align}   \begin{split}   \int_{t_1}^{t_1+1} \Vert \grad u \Vert_{L^\infty}   &\leq    \int_{t_1}^{t_1+1}      \Bigl(      \Vert \grad u \Vert_{L^2}^{1/4}      \Vert \Delta u \Vert_{L^3}^{3/4}      +      \Vert \grad u \Vert_{L^2}          \Bigr)    \\&   \leq    C     \left( \int_{t_1}^{t_1+1} \Vert \grad u \Vert_{L^2}^{1/3} \right)^{3/4}     \left( \int_{t_1}^{t_1+1} \Vert \Delta u \Vert_{L^3}^3  \right)^{1/4}     + \frac12\epsilon    ,   \end{split}    \label{EWRTDFBSERSDFGSDFHDSFADSFSDFSVBASWERTDFSGSDFGFDGASFASDFDSFHFDGHFH47}   \end{align} provided $t_0$ is sufficiently large. To bound the $L^3 L^3$ norm of $\Delta u$, we introduce a smooth cut-off function  $\phi\colon [0,\infty) \to [0,1]$, where $\phi(t)=0$ on $[0,t_1-1]$ and $\phi(t)=1$ on $[t_1,\infty]$ with $|\phi'|\les 1$. Now we consider the equation   \begin{equation}    (\phi u)_{t}     - \Delta(\phi u)     + \nabla (\phi p)     = \phi' u       - u\cdot \nabla(\phi u)       + \phi \rho e_2    \llabel{ Cm m P3d WSb jst oj 3oGj advz qcMB6Y 6k D 9sZ 0bd Mjt UT hULG TWU9 Nmr3E4 CN b zUO vTh hqL 1p xAxT ezrH dVMgLY TT r Sfx LUX CMr WA bE69 K6XH i5re1f x4 G DKk iB7 f2D Xz Xez2 k2Yc Yc4QjU yM Y R1o DeY NWf 74 hByF dsWk 4cUbCR DX a q4e DWd 7qb Ot 7GOu oklg jJ00J9 Il O Jxn tzF VBC Ft pABp VLEE 2y5Qcg b3 5 DU4 igj 4dz zW soNF wvqj bNFma0 am F Kiv EWRTDFBSERSDFGSDFHDSFADSFSDFSVBASWERTDFSGSDFGFDGASFASDFDSFHFDGHFH164}   \end{equation} which follows from \eqref{EWRTDFBSERSDFGSDFHDSFADSFSDFSVBASWERTDFSGSDFGFDGASFASDFDSFHFDGHFH01}$_1$; note that $\nabla \cdot(\phi u)=0$ since $\phi$ is a function of time only. Using the $W^{2,3}$ estimate due to Sohr and Von~Wahl \cite{SvW} we have, similarly to \eqref{EWRTDFBSERSDFGSDFHDSFADSFSDFSVBASWERTDFSGSDFGFDGASFASDFDSFHFDGHFH167}--\eqref{EWRTDFBSERSDFGSDFHDSFADSFSDFSVBASWERTDFSGSDFGFDGASFASDFDSFHFDGHFH170},   \begin{align}   \begin{split}   \int_{t_1}^{t_1+1} \Vert D^2u \Vert_{L^3}^3   & \les   \int_{t_1-1}^{t_1+1} \Vert u \cdot \grad(\phi u) \Vert_{L^3}^3   +   \int_{t_1-1}^{t_1+1} \Vert \phi' u \Vert_{L^3}^3   +   \int_{t_1-1}^{t_1+1} \Vert \rho \Vert_{L^3}^3   \\&   \les   \int_{t_1-1}^{t_1+1} \Vert u \Vert_{L^6}^3\Vert \grad u \Vert_{L^6}^3   +   \int_{t_1-1}^{t_1+1} \Vert u \Vert_{L^3}^3   +   1   \\&   \les   \int_{t_1-1}^{t_1+1} \Vert u \Vert_{L^2}   \Vert \grad u \Vert_{L^2}^3   \Vert Au \Vert_{L^2}^2   +   \int_{t_1-1}^{t_1+1} \Vert u \Vert_{L^4}^4   +   1   \les   1    \comma t\geq0   \end{split}    \label{EWRTDFBSERSDFGSDFHDSFADSFSDFSVBASWERTDFSGSDFGFDGASFASDFDSFHFDGHFH48}   \end{align} where we used \eqref{EWRTDFBSERSDFGSDFHDSFADSFSDFSVBASWERTDFSGSDFGFDGASFASDFDSFHFDGHFH08}, \eqref{EWRTDFBSERSDFGSDFHDSFADSFSDFSVBASWERTDFSGSDFGFDGASFASDFDSFHFDGHFH17}, and \eqref{EWRTDFBSERSDFGSDFHDSFADSFSDFSVBASWERTDFSGSDFGFDGASFASDFDSFHFDGHFH122}. Also, for the first factor of the first term in \eqref{EWRTDFBSERSDFGSDFHDSFADSFSDFSVBASWERTDFSGSDFGFDGASFASDFDSFHFDGHFH47}, we use \eqref{EWRTDFBSERSDFGSDFHDSFADSFSDFSVBASWERTDFSGSDFGFDGASFASDFDSFHFDGHFH09} to obtain that  for any $\epsilon_0>0$ there exists $t_0\geq1$ sufficiently large so that   \begin{equation}     \left( \int_{t_1-1}^{t_1+1} \Vert \grad u \Vert_{L^2}^{1/3} dt\right)^{3/4}      \leq       \epsilon_0 \epsilon     .    \label{EWRTDFBSERSDFGSDFHDSFADSFSDFSVBASWERTDFSGSDFGFDGASFASDFDSFHFDGHFH50}   \end{equation} Thus, using \eqref{EWRTDFBSERSDFGSDFHDSFADSFSDFSVBASWERTDFSGSDFGFDGASFASDFDSFHFDGHFH48} and \eqref{EWRTDFBSERSDFGSDFHDSFADSFSDFSVBASWERTDFSGSDFGFDGASFASDFDSFHFDGHFH50} in \eqref{EWRTDFBSERSDFGSDFHDSFADSFSDFSVBASWERTDFSGSDFGFDGASFASDFDSFHFDGHFH47}, we obtain   \begin{equation}     \int_{t_1}^{t_1+1} \Vert \grad u \Vert_{L^\infty} dt      \leq C \epsilon_0\epsilon  + \frac12 \epsilon    \comma t\geq t_0    ,    \llabel{Aap pzM zr VqYf OulM HafaBk 6J r eOQ BaT EsJ BB tHXj n2EU CNleWp cv W JIg gWX Ksn B3 wvmo WK49 Nl492o gR 6 fvc 8ff jJm sW Jr0j zI9p CBsIUV of D kKH Ub7 vxp uQ UXA6 hMUr yvxEpc Tq l Tkz z0q HbX pO 8jFu h6nw zVPPzp A8 9 61V 78c O2W aw 0yGn CHVq BVjTUH lk p 6dG HOd voE E8 cw7Q DL1o 1qg5TX qo V 720 hhQ TyF tp TJDg 9E8D nsp1Qi X9 8 ZVQ N3s duZ qcEWRTDFBSERSDFGSDFHDSFADSFSDFSVBASWERTDFSGSDFGFDGASFASDFDSFHFDGHFH51}   \end{equation} for $t_0\geq1$ sufficiently large, which in turn implies   \begin{equation}    \int_{t_0}^t \Vert \grad u \Vert_{L^\infty} dt \leq \epsilon (t-t_0)    \comma t\geq t_0       \label{EWRTDFBSERSDFGSDFHDSFADSFSDFSVBASWERTDFSGSDFGFDGASFASDFDSFHFDGHFH165}   \end{equation} if we choose $\epsilon_0$ a sufficiently small constant. Note that \eqref{EWRTDFBSERSDFGSDFHDSFADSFSDFSVBASWERTDFSGSDFGFDGASFASDFDSFHFDGHFH165} is obtained by adding the integrals of unit length. Returning to \eqref{EWRTDFBSERSDFGSDFHDSFADSFSDFSVBASWERTDFSGSDFGFDGASFASDFDSFHFDGHFH46}, we find that Gronwall's inequality implies   \begin{align}   \begin{split}   \Vert \grad \theta(t) \Vert_{L^2}    \leq    (\Vert \grad \theta(t_0) \Vert_{L^2} +1)   e^{\epsilon (t-t_0)}   .   \end{split}   \label{EWRTDFBSERSDFGSDFHDSFADSFSDFSVBASWERTDFSGSDFGFDGASFASDFDSFHFDGHFH52}   \end{align} Finally, we use \eqref{EWRTDFBSERSDFGSDFHDSFADSFSDFSVBASWERTDFSGSDFGFDGASFASDFDSFHFDGHFH160} implying   \begin{equation}    \Vert \theta(t_0)\Vert_{H^{1}}    \les 1    ,    \label{EWRTDFBSERSDFGSDFHDSFADSFSDFSVBASWERTDFSGSDFGFDGASFASDFDSFHFDGHFH162}   \end{equation} where the constant depends on $t_0$, which in turn only depends on $\epsilon$. Combining \eqref{EWRTDFBSERSDFGSDFHDSFADSFSDFSVBASWERTDFSGSDFGFDGASFASDFDSFHFDGHFH52} and \eqref{EWRTDFBSERSDFGSDFHDSFADSFSDFSVBASWERTDFSGSDFGFDGASFASDFDSFHFDGHFH162} leads to the claimed inequality \eqref{EWRTDFBSERSDFGSDFHDSFADSFSDFSVBASWERTDFSGSDFGFDGASFASDFDSFHFDGHFH43}. \end{proof} \par We noted that the initial assumptions of Theorem~\ref{T01} can be relaxed,  implying the conclusions of Theorem~\ref{T01} for  less restrictive initial conditions than those  required for \eqref{EWRTDFBSERSDFGSDFHDSFADSFSDFSVBASWERTDFSGSDFGFDGASFASDFDSFHFDGHFH17}. \par \cole \begin{thm} \label{T03} Let $(u_0,\rho_0) \in V \times H^1(\Omega)$. Then there exists a unique solution $(u,\rho)$ such that $u \in L_{\loc}^2([0,\infty);D(A)) \cap C([0,\infty);H)$ and $\rho \in L_{\loc}^\infty([0,\infty);H^1(\Omega))$, which moreover satisfies    \begin{equation}    \Vert Au\Vert_{L^2}\leq C_{\delta}       \comma t\geq \delta    ,    \label{EWRTDFBSERSDFGSDFHDSFADSFSDFSVBASWERTDFSGSDFGFDGASFASDFDSFHFDGHFH125}   \end{equation}  where $\delta>0$ is arbitrary. \end{thm} \colb \par \begin{proof}[Proof of Theorem~\ref{T03}] Let $(u,\rho)$ be a solution to \eqref{EWRTDFBSERSDFGSDFHDSFADSFSDFSVBASWERTDFSGSDFGFDGASFASDFDSFHFDGHFH20} on $[0,T]$ where $T\in(0,1]$. Integrating \eqref{EWRTDFBSERSDFGSDFHDSFADSFSDFSVBASWERTDFSGSDFGFDGASFASDFDSFHFDGHFH21} in time, we obtain   \begin{align}   \Vert u(t) \Vert_{L^2}^2 +\Vert \theta(t) \Vert_{L^2}^2    + \int_0^t \Vert \grad u \Vert_{L^2}^2   \les   \Vert u_0 \Vert_{L^2}^2 + \Vert \theta_0 \Vert_{L^2}^2   \les   1    \comma t\in[0,T].   \label{EWRTDFBSERSDFGSDFHDSFADSFSDFSVBASWERTDFSGSDFGFDGASFASDFDSFHFDGHFH54}  \end{align} We note that all constants are allowed to depend on $\Vert u_0\Vert_{V}$ and $\Vert \rho_0\Vert_{H^{1}}$. We use this inequality in \eqref{EWRTDFBSERSDFGSDFHDSFADSFSDFSVBASWERTDFSGSDFGFDGASFASDFDSFHFDGHFH28} obtaining   \begin{align}   \begin{split}   \frac{d}{dt} \Vert A^{1/2}u\Vert_{L^2}^2 + \Vert Au \Vert_{L^2}^2   \les   \Vert A^{1/2}u \Vert_{L^2}^4   + 1   ,   \end{split}   \label{EWRTDFBSERSDFGSDFHDSFADSFSDFSVBASWERTDFSGSDFGFDGASFASDFDSFHFDGHFH56}   \end{align} which implies, along with \eqref{EWRTDFBSERSDFGSDFHDSFADSFSDFSVBASWERTDFSGSDFGFDGASFASDFDSFHFDGHFH54}, that upon suitably reducing $T>0$, we have $u \in L^\infty([0,T] ; V) \cap L^2([0,T] ; D(A))$ and   \begin{equation}    \Vert A^{1/2}u\Vert_{L^2}^2    \les    1    \comma t\in[0,T]    ,    \label{EWRTDFBSERSDFGSDFHDSFADSFSDFSVBASWERTDFSGSDFGFDGASFASDFDSFHFDGHFH124}   \end{equation} and then    \begin{equation}    \int_{0}^{T}    \Vert A u\Vert_{L^2}^2    \les    1    \comma t\in[0,T]    ,    \label{EWRTDFBSERSDFGSDFHDSFADSFSDFSVBASWERTDFSGSDFGFDGASFASDFDSFHFDGHFH129}   \end{equation} upon returning to \eqref{EWRTDFBSERSDFGSDFHDSFADSFSDFSVBASWERTDFSGSDFGFDGASFASDFDSFHFDGHFH56}. Note that   \begin{align}    \begin{split}    \Vert u_t\Vert_{L^2}^2    &\les    \Vert A u\Vert_{L^2}^2       +    \Vert B(u,u)\Vert_{L^2}^2       +    \Vert \rho\Vert_{L^2}^2    \les    \Vert A u\Vert_{L^2}^2       + \Vert u\Vert_{L^2}      \Vert A^{1/2}u\Vert_{L^2}^2      \Vert A u\Vert_{L^2}         + 1    \\&    \les    \Vert A u\Vert_{L^2}^2       + \Vert u\Vert_{L^2}^2      \Vert A^{1/2}u\Vert_{L^2}^4    + 1    \les    \Vert A u\Vert_{L^2}^2   + 1    ,    \end{split}    \label{EWRTDFBSERSDFGSDFHDSFADSFSDFSVBASWERTDFSGSDFGFDGASFASDFDSFHFDGHFH31}   \end{align} by \eqref{EWRTDFBSERSDFGSDFHDSFADSFSDFSVBASWERTDFSGSDFGFDGASFASDFDSFHFDGHFH17}, \eqref{EWRTDFBSERSDFGSDFHDSFADSFSDFSVBASWERTDFSGSDFGFDGASFASDFDSFHFDGHFH131}, and \eqref{EWRTDFBSERSDFGSDFHDSFADSFSDFSVBASWERTDFSGSDFGFDGASFASDFDSFHFDGHFH122}. From \eqref{EWRTDFBSERSDFGSDFHDSFADSFSDFSVBASWERTDFSGSDFGFDGASFASDFDSFHFDGHFH129} and \eqref{EWRTDFBSERSDFGSDFHDSFADSFSDFSVBASWERTDFSGSDFGFDGASFASDFDSFHFDGHFH31}, we obtain $u_t \in L^2([0,T]; H)$. Thus, we may modify $u$ on a measure zero subset of $[0,T]$ so that $u \in C([0,T]; H)$. \par In order to prove \eqref{EWRTDFBSERSDFGSDFHDSFADSFSDFSVBASWERTDFSGSDFGFDGASFASDFDSFHFDGHFH125}, we first need to show uniqueness in the class $V\times H^{1}(\Omega)$. Thus, let $(\uone, \thetaone)$ and $(\utwo,\thetatwo)$ be solutions to the Boussinesq equation, and define $u = \uone - \utwo$ and $\rho= \rhoone-\rhotwo$ with both solutions satisfying the bounds \eqref{EWRTDFBSERSDFGSDFHDSFADSFSDFSVBASWERTDFSGSDFGFDGASFASDFDSFHFDGHFH124} and \eqref{EWRTDFBSERSDFGSDFHDSFADSFSDFSVBASWERTDFSGSDFGFDGASFASDFDSFHFDGHFH129} on $[0,T]$. Then subtracting the evolution equations \eqref{EWRTDFBSERSDFGSDFHDSFADSFSDFSVBASWERTDFSGSDFGFDGASFASDFDSFHFDGHFH01}$_1$ for $\uone$ and $\utwo$ and testing the equation for the difference with  $Au$, we acquire   \begin{align}   \begin{split}   &\frac12\frac{d}{dt} \Vert A^{1/2} u \Vert_{L^2}^2    + \Vert Au \Vert_{L^2}^2   \\&\indeq   \les    \Vert \uone\Vert_{L^2}^{1/2}    \Vert A^{1/2}\uone\Vert_{L^2}^{1/2}    \Vert A^{1/2}u\Vert_{L^2}^{1/2}    \Vert Au\Vert_{L^2}^{3/2}    \\&\indeq\indeq    +    \Vert u\Vert_{L^2}^{1/2}    \Vert A^{1/2}u\Vert_{L^2}^{1/2}    \Vert A^{1/2}\utwo\Vert_{L^2}^{1/2}    \Vert A\utwo\Vert_{L^2}^{1/2}    \Vert Au\Vert_{L^2}
   +    \Vert \rho\Vert_{L^2}    \Vert Au\Vert_{L^2}   ,   \end{split}    \label{EWRTDFBSERSDFGSDFHDSFADSFSDFSVBASWERTDFSGSDFGFDGASFASDFDSFHFDGHFH57}   \end{align} whence, using the bounds on $\uone$ and $\utwo$ and absorbing factors of  $\Vert Au\Vert_{L^2}$, we get   \begin{align}   \begin{split}   &\frac{d}{dt} \Vert A^{1/2} u \Vert_{L^2}^2    + \Vert Au \Vert_{L^2}^2   \les    \Vert A^{1/2}u\Vert_{L^2}^{2}    +    \Vert u\Vert_{L^2}    \Vert A^{1/2}u\Vert_{L^2}    \Vert A\utwo\Vert_{L^2}    +    \Vert \rho\Vert_{L^2}^2    .   \end{split}   \llabel{ n9IX ozWh Fd16IB 0K 9 JeB Hvi 364 kQ lFMM JOn0 OUBrnv pY y jUB Ofs Pzx l4 zcMn JHdq OjSi6N Mn 8 bR6 kPe klT Fd VlwD SrhT 8Qr0sC hN h 88j 8ZA vvW VD 03wt ETKK NUdr7W EK 1 jKS IHF Kh2 sr 1RRV Ra8J mBtkWI 1u k uZT F2B 4p8 E7 Y3p0 DX20 JM3XzQ tZ 3 bMC vM4 DEA wB Fp8q YKpL So1a5s dR P fTg 5R6 7v1 T4 eCJ1 qg14 CTK7u7 ag j Q0A tZ1 Nh6 hk Sys5 CWonEWRTDFBSERSDFGSDFHDSFADSFSDFSVBASWERTDFSGSDFGFDGASFASDFDSFHFDGHFH130}   \end{align} On the other hand, from the density equations for $\rhoone$ and $\rhotwo$, we get   \begin{align}    \begin{split}    \frac{d}{dt}\Vert \rho\Vert_{L^2}^2     &    \les    \Vert u\Vert_{L^\infty}    \Vert \nabla\rhotwo\Vert_{L^2}    \Vert\rho\Vert_{L^2}    \les    \Vert u\Vert_{L^2}^{1/2}    \Vert A u\Vert_{L^2}^{1/2}    \Vert \nabla\rhotwo\Vert_{L^2}    \Vert\rho\Vert_{L^2}    \\&    \les    \epsilon_0    \Vert A u\Vert_{L^2}^2    +    \Vert u\Vert_{L^2}^{2/3}    \Vert \nabla\rhotwo\Vert_{L^2}^{4/3}    \Vert\rho\Vert_{L^2}^{4/3}    \\&    \les    \epsilon_0    \Vert A u\Vert_{L^2}^2    +    (\Vert A^{1/2}u\Vert_{L^2}^{2} + \Vert\rho\Vert_{L^2}^{2})    \Vert \nabla\rhotwo\Vert_{L^2}^{4/3}    ,    \end{split}    \label{EWRTDFBSERSDFGSDFHDSFADSFSDFSVBASWERTDFSGSDFGFDGASFASDFDSFHFDGHFH127}   \end{align} where $\epsilon_0$ is a sufficiently small constant to be determined. Adding \eqref{EWRTDFBSERSDFGSDFHDSFADSFSDFSVBASWERTDFSGSDFGFDGASFASDFDSFHFDGHFH57} and \eqref{EWRTDFBSERSDFGSDFHDSFADSFSDFSVBASWERTDFSGSDFGFDGASFASDFDSFHFDGHFH127}, choosing $\epsilon_0$ sufficiently small and absorbing factors of $\Vert Au\Vert_{L^2}^2$, we obtain   \begin{align}    \begin{split}    &\frac{d}{dt}       (\Vert A^{1/2} u \Vert_{L^2}^2        +  \Vert \rho\Vert_{L^2}^2      )    \les    \Vert A^{1/2}u\Vert_{L^2}^{2}    +    \Vert A\utwo\Vert_{L^2}    \Vert A^{1/2}u\Vert_{L^2}^2    +       \Vert u\Vert_{L^2}^2    +    \Vert\rho\Vert_{L^2}^2    \colb    ,    \end{split}    \label{EWRTDFBSERSDFGSDFHDSFADSFSDFSVBASWERTDFSGSDFGFDGASFASDFDSFHFDGHFH128}   \end{align} where we used $\Vert \nabla\rhotwo\Vert_{L^2}\les 1$ for $t\in[0,T]$ on the last term in \eqref{EWRTDFBSERSDFGSDFHDSFADSFSDFSVBASWERTDFSGSDFGFDGASFASDFDSFHFDGHFH127}, subject to reducing $T$. Applying a Gronwall argument to \eqref{EWRTDFBSERSDFGSDFHDSFADSFSDFSVBASWERTDFSGSDFGFDGASFASDFDSFHFDGHFH128} and using \eqref{EWRTDFBSERSDFGSDFHDSFADSFSDFSVBASWERTDFSGSDFGFDGASFASDFDSFHFDGHFH129} for $\utwo$, we conclude that  $u=0$, whence $\uone=\utwo$ on $[0,T]$. \par In order to obtain \eqref{EWRTDFBSERSDFGSDFHDSFADSFSDFSVBASWERTDFSGSDFGFDGASFASDFDSFHFDGHFH125}, we observe that $u(t) \in D(A)$ for a.e. $t \in [0,T]$ by \eqref{EWRTDFBSERSDFGSDFHDSFADSFSDFSVBASWERTDFSGSDFGFDGASFASDFDSFHFDGHFH129}. We choose  $t_0 \in (0,\delta)$ such that $u(t_0) \in D(A)$. Since $u$ is unique on $[t_0,\infty)$, we may apply Theorem~\ref{T01} and obtain \eqref{EWRTDFBSERSDFGSDFHDSFADSFSDFSVBASWERTDFSGSDFGFDGASFASDFDSFHFDGHFH125} for $t\geq t_0$, concluding the proof. \end{proof} \par We remark that similar arguments show analogous reduced required regularity for $u$ in Theorems \ref{T02} and~\ref{T04}. \par \colb Next, we address a higher regularity norm. \par \begin{proof}[Proof of Theorem~\ref{T02}]  We start with a~priori estimates and at the end of the proof we provide a sketch of the justification. Taking a time derivative of \eqref{EWRTDFBSERSDFGSDFHDSFADSFSDFSVBASWERTDFSGSDFGFDGASFASDFDSFHFDGHFH01}$_1$, we  obtain   \begin{align}    \begin{split}    u_{tt}       - \Laplace u_t       + u_t \cdot \grad u       + u \cdot \grad u_t      + \grad p_t       = \rho_t e_2     ,    \end{split}    \llabel{ IOqgCL 3u 7 feR BHz odS Jp 7JH8 u6Rw sYE0mc P4 r LaW Atl yRw kH F3ei UyhI iA19ZB u8 m ywf 42n uyX 0e ljCt 3Lkd 1eUQEZ oO Z rA2 Oqf oQ5 Ca hrBy KzFg DOseim 0j Y BmX csL Ayc cC JBTZ PEjy zPb5hZ KW O xT6 dyt u82 Ia htpD m75Y DktQvd Nj W jIQ H1B Ace SZ KVVP 136v L8XhMm 1O H Kn2 gUy kFU wN 8JML Bqmn vGuwGR oW U oNZ Y2P nmS 5g QMcR YHxL yHuDo8 baEWRTDFBSERSDFGSDFHDSFADSFSDFSVBASWERTDFSGSDFGFDGASFASDFDSFHFDGHFH59}   \end{align} which, after testing  with $u_{tt}$ gives   \begin{align}    \begin{split}    &    \frac12 \frac{d}{dt}      \Vert \nabla u_t\Vert_{L^2}^2      + \Vert u_{tt}\Vert_{L^2}^2      \\&\indeq      =       - \int_{\Omega} u_t\cdot \nabla u_j \FGSDFHGFHDFGHDFGH_{tt} u_j       - \int_{\Omega} u\cdot \nabla \FGSDFHGFHDFGHDFGH_{t}u_j \FGSDFHGFHDFGHDFGH_{tt} u_j       + \int_{\Omega} (\rho_t e_2)\cdot u_{tt}     \\&\indeq     \les     \Vert u_t\Vert_{L^2}^{1/2}     \Vert \nabla u_{t}\Vert_{L^2}^{1/2}     \Vert \nabla u\Vert_{L^2}^{1/2}     \Vert A u\Vert_{L^2}^{1/2}     \Vert u_{tt}\Vert_{L^2}     +     \Vert u\Vert_{L^{\infty}}     \Vert \nabla u_t\Vert_{L^2}     \Vert u_{tt}\Vert_{L^2}     +     \Vert \rho_t\Vert_{L^2}     \Vert u_{tt}\Vert_{L^2}         .    \end{split}    \llabel{ w aqM NYt onW u2 YIOz eB6R wHuGcn fi o 47U PM5 tOj sz QBNq 7mco fCNjou 83 e mcY 81s vsI 2Y DS3S yloB Nx5FBV Bc 9 6HZ EOX UO3 W1 fIF5 jtEM W6KW7D 63 t H0F CVT Zup Pl A9aI oN2s f1Bw31 gg L FoD O0M x18 oo heEd KgZB Cqdqpa sa H Fhx BrE aRg Au I5dq mWWB MuHfv9 0y S PtG hFF dYJ JL f3Ap k5Ck Szr0Kb Vd i sQk uSA JEn DT YkjP AEMu a0VCtC Ff z 9R6 VhtEWRTDFBSERSDFGSDFHDSFADSFSDFSVBASWERTDFSGSDFGFDGASFASDFDSFHFDGHFH60}   \end{align} Now we apply $\Vert u\Vert_{L^\infty}\les \Vert u\Vert_{L^2}^{1/2}\Vert Au\Vert_{L^2}^{1/2}$ for the second term and   \begin{equation}    \Vert \rho_t\Vert_{L^2}=\Vert u\cdot \nabla \rho\Vert_{L^2}    \les \Vert u\Vert_{L^\infty}\Vert\nabla \rho\Vert_{L^2}       ,    \llabel{ 8Ua cB e7op AnGa 7AbLWj Hc s nAR GMb n7a 9n paMf lftM 7jvb20 0T W xUC 4lt e92 9j oZrA IuIa o1Zqdr oC L 55L T4Q 8kN yv sIzP x4i5 9lKTq2 JB B sZb QCE Ctw ar VBMT H1QR 6v5srW hR r D4r wf8 ik7 KH Egee rFVT ErONml Q5 L R8v XNZ LB3 9U DzRH ZbH9 fTBhRw kA 2 n3p g4I grH xd fEFu z6RE tDqPdw N7 H TVt cE1 8hW 6y n4Gn nCE3 MEQ51i Ps G Z2G Lbt CSt hu zvEWRTDFBSERSDFGSDFHDSFADSFSDFSVBASWERTDFSGSDFGFDGASFASDFDSFHFDGHFH55}   \end{equation} by \eqref{EWRTDFBSERSDFGSDFHDSFADSFSDFSVBASWERTDFSGSDFGFDGASFASDFDSFHFDGHFH01}$_2$, on the last. Absorbing the factors of $\Vert u_{tt}\Vert_{L^2}$, we obtain   \begin{align}    \begin{split}    &    \frac12 \frac{d}{dt}      \Vert \nabla u_t\Vert_{L^2}^2      + \Vert u_{tt}\Vert_{L^2}^2     \\&\indeq     \les     \Vert u_t\Vert_{L^2}     \Vert \nabla u_t\Vert_{L^2}     \Vert \nabla u\Vert_{L^2}     \Vert A u\Vert_{L^2}     +     \Vert u\Vert_{L^{\infty}}^2     \Vert \nabla u_t\Vert_{L^2}^2     +     \Vert u\Vert_{L^\infty}^2     \Vert \nabla \rho\Vert_{L^2}^2     \\&\indeq     \les      1 + \Vert \nabla u_t\Vert_{L^2}^2     +  C_{\epsilon} e^{2\epsilon t}    ,    \end{split}    \label{EWRTDFBSERSDFGSDFHDSFADSFSDFSVBASWERTDFSGSDFGFDGASFASDFDSFHFDGHFH61}   \end{align} where $\epsilon>0$ is arbitrarily small. In \eqref{EWRTDFBSERSDFGSDFHDSFADSFSDFSVBASWERTDFSGSDFGFDGASFASDFDSFHFDGHFH61}, we also used \eqref{EWRTDFBSERSDFGSDFHDSFADSFSDFSVBASWERTDFSGSDFGFDGASFASDFDSFHFDGHFH08}. Combining \eqref{EWRTDFBSERSDFGSDFHDSFADSFSDFSVBASWERTDFSGSDFGFDGASFASDFDSFHFDGHFH38} and \eqref{EWRTDFBSERSDFGSDFHDSFADSFSDFSVBASWERTDFSGSDFGFDGASFASDFDSFHFDGHFH61} with a uniform Gronwall argument, we get   \begin{align}    \Vert \nabla u_t \Vert_{L^2}    \les     e^{\epsilon t}    \comma t\geq0       \label{EWRTDFBSERSDFGSDFHDSFADSFSDFSVBASWERTDFSGSDFGFDGASFASDFDSFHFDGHFH62}   \end{align} and   \begin{equation}     \int_{0}^{t} \Vert u_{tt}\Vert_{L^2}^2    \les     e^{\epsilon t}    \comma t\geq 0    ,    \llabel{PF eE28 MM23ug TC d j7z 7Av TLa 1A GLiJ 5JwW CiDPyM qa 8 tAK QZ9 cfP 42 kuUz V3h6 GsGFoW m9 h cfj 51d GtW yZ zC5D aVt2 Wi5IIs gD B 0cX LM1 FtE xE RIZI Z0Rt QUtWcU Cm F mSj xvW pZc gl dopk 0D7a EouRku Id O ZdW FOR uqb PY 6HkW OVi7 FuVMLW nx p SaN omk rC5 uI ZK9C jpJy UIeO6k gb 7 tr2 SCY x5F 11 S6Xq OImr s7vv0u vA g rb9 hGP Fnk RM j92H gczJ 66EWRTDFBSERSDFGSDFHDSFADSFSDFSVBASWERTDFSGSDFGFDGASFASDFDSFHFDGHFH63}   \end{equation} where we allow constants to depend on $\epsilon$. Now, consider the stationary, i.e., pointwise in time, Stokes problem   \begin{align}    \begin{split}    &    -\Delta u + \nabla p = - u\cdot\nabla u - u_t + \rho e_2     \\&     u|_{\FGSDFHGFHDFGHDFGH\Omega} = 0     .    \end{split}    \label{EWRTDFBSERSDFGSDFHDSFADSFSDFSVBASWERTDFSGSDFGFDGASFASDFDSFHFDGHFH156}   \end{align} Note that   \begin{align}    \begin{split}    &    \Vert  u\cdot\nabla u + u_t - \rho e_2 \Vert_{H^{1}}    \les    \Vert D(u\cdot\nabla u - u_t)\Vert_{L^2}       + \Vert \rho\Vert_{H^{1}}    \\&\indeq    \les    \Vert D u\Vert_{L^4}^2    + \Vert u\Vert_{L^\infty}       \Vert D^2 u\Vert_{L^2}    + \Vert \nabla u_t\Vert_{L^2}    + e^{\epsilon t}    \\&\indeq    \les    \Vert A^{1/2}u\Vert_{L^2}    \Vert Au\Vert_{L^2}    + \Vert u\Vert_{L^\infty}       \Vert A u\Vert_{L^2}    + \Vert \nabla u_t\Vert_{L^2}    + e^{\epsilon t}    \les    e^{\epsilon t}    ,    \end{split}    \llabel{0kHb BB l QSI OY7 FcX 0c uyDl LjbU 3F6vZk Gb a KaM ufj uxp n4 Mi45 7MoL NW3eIm cj 6 OOS e59 afA hg lt9S BOiF cYQipj 5u N 19N KZ5 Czc 23 1wxG x1ut gJB4ue Mx x 5lr s8g VbZ s1 NEfI 02Rb pkfEOZ E4 e seo 9te NRU Ai nujf eJYa Ehns0Y 6X R UF1 PCf 5eE AL 9DL6 a2vm BAU5Au DD t yQN 5YL LWw PW GjMt 4hu4 FIoLCZ Lx e BVY 5lZ DCD 5Y yBwO IJeH VQsKob Yd q EWRTDFBSERSDFGSDFHDSFADSFSDFSVBASWERTDFSGSDFGFDGASFASDFDSFHFDGHFH157}   \end{align} using \eqref{EWRTDFBSERSDFGSDFHDSFADSFSDFSVBASWERTDFSGSDFGFDGASFASDFDSFHFDGHFH62} in the last step. Applying the $H^{3}$ regularity for the Stokes problem \eqref{EWRTDFBSERSDFGSDFHDSFADSFSDFSVBASWERTDFSGSDFGFDGASFASDFDSFHFDGHFH156}, cf.~\cite[Proposition~3.3]{T4}, leads to   \begin{equation}    \Vert u\Vert_{H^{3}}    + \Vert \nabla p\Vert_{H^{1}}    \les    e^{\epsilon t}     .    \label{EWRTDFBSERSDFGSDFHDSFADSFSDFSVBASWERTDFSGSDFGFDGASFASDFDSFHFDGHFH166}   \end{equation} In order to obtain \eqref{EWRTDFBSERSDFGSDFHDSFADSFSDFSVBASWERTDFSGSDFGFDGASFASDFDSFHFDGHFH14},  we apply $\FGSDFHGFHDFGHDFGH_{ij}$, for $i,j=1,2$, to \eqref{EWRTDFBSERSDFGSDFHDSFADSFSDFSVBASWERTDFSGSDFGFDGASFASDFDSFHFDGHFH01}$_2$, test  with $\FGSDFHGFHDFGHDFGH_{ij}\rho$,  and sum which leads to   \begin{align}   \begin{split}   \frac{1}{2} \frac{d}{dt} \Vert \FGSDFHGFHDFGHDFGH_{ij} \rho \Vert_{L^2}^2    &=   \langle \FGSDFHGFHDFGHDFGH_{ij} (u \cdot \grad \rho),\FGSDFHGFHDFGHDFGH_{ij} \rho\rangle_{L^2}    =   \int_{\Omega} \FGSDFHGFHDFGHDFGH_{ij} u_k \FGSDFHGFHDFGHDFGH_k \rho \FGSDFHGFHDFGHDFGH_{ij} \rho     + 2\int_{\Omega} \FGSDFHGFHDFGHDFGH_{i} u_k \FGSDFHGFHDFGHDFGH_{jk} \rho \FGSDFHGFHDFGHDFGH_{ij} \rho     + \int_{\Omega} u_k \FGSDFHGFHDFGHDFGH_{ijk}\rho \FGSDFHGFHDFGHDFGH_{ij} \rho   ,   \end{split}   \label{EWRTDFBSERSDFGSDFHDSFADSFSDFSVBASWERTDFSGSDFGFDGASFASDFDSFHFDGHFH67}   \end{align} which holds for all $t\geq0$. The last term vanishes due to the incompressibility, while the second is bounded by $C\Vert \grad u \Vert_{L^\infty} \Vert D^2 \rho \Vert_{L^2}^2$.  For the first term on the far right side of \eqref{EWRTDFBSERSDFGSDFHDSFADSFSDFSVBASWERTDFSGSDFGFDGASFASDFDSFHFDGHFH67}, we write   \begin{align}   \begin{split}   &\int_{\Omega} \FGSDFHGFHDFGHDFGH_{ij} u_k \FGSDFHGFHDFGHDFGH_k \rho \FGSDFHGFHDFGHDFGH_{ij} \rho    \les   \Vert \Delta u \Vert_{L^{4}}    \Vert \grad \rho \Vert_{L^4}    \Vert D^2 \rho \Vert_{L^2}   \\&\indeq   \les    (   \Vert \Delta u \Vert_{L^2}^{1/2}   \Vert D^{3} u \Vert_{L^2}^{1/2}   +   \Vert \Delta u \Vert_{L^2}   )   (   \Vert \grad \rho \Vert_{L^2}^{1/2}    \Vert D^2 \rho \Vert_{L^2}^{1/2}   +   \Vert \nabla\rho\Vert_{L^2}   )   \Vert D^2 \rho \Vert_{L^2}   ,   \end{split}    \label{EWRTDFBSERSDFGSDFHDSFADSFSDFSVBASWERTDFSGSDFGFDGASFASDFDSFHFDGHFH68}   \end{align} where we utilized the Gagliardo-Nirenberg inequalities. Now, we use \eqref{EWRTDFBSERSDFGSDFHDSFADSFSDFSVBASWERTDFSGSDFGFDGASFASDFDSFHFDGHFH11} and \eqref{EWRTDFBSERSDFGSDFHDSFADSFSDFSVBASWERTDFSGSDFGFDGASFASDFDSFHFDGHFH166} in \eqref{EWRTDFBSERSDFGSDFHDSFADSFSDFSVBASWERTDFSGSDFGFDGASFASDFDSFHFDGHFH68}, sum in $i$ and~$j$,  and cancel a factor of $\Vert D^2 \rho\Vert_{L^2}$ on both sides to obtain   \begin{align}    \begin{split}    \frac{1}{2} \frac{d}{dt} \Vert D^2 \rho \Vert_{L^2}    &\les    e^{3\epsilon t/2}    +    C_{\epsilon}    e^{\epsilon t}       \Vert D^2 \rho\Vert_{L^2}^{1/2}    +    \Vert \nabla u\Vert_{L^\infty}    \Vert D^2 \rho\Vert_{L^2}       ,       \end{split}    \llabel{fCX 1to mCb Ej 5m1p Nx9p nLn5A3 g7 U v77 7YU gBR lN rTyj shaq BZXeAF tj y FlW jfc 57t 2f abx5 Ns4d clCMJc Tl q kfq uFD iSd DP eX6m YLQz JzUmH0 43 M lgF edN mXQ Pj Aoba 07MY wBaC4C nj I 4dw KCZ PO9 wx 3en8 AoqX 7JjN8K lq j Q5c bMS dhR Fs tQ8Q r2ve 2HT0uO 5W j TAi iIW n1C Wr U1BH BMvJ 3ywmAd qN D LY8 lbx XMx 0D Dvco 3RL9 Qz5eqy wV Y qEN nO8 MHEWRTDFBSERSDFGSDFHDSFADSFSDFSVBASWERTDFSGSDFGFDGASFASDFDSFHFDGHFH66}   \end{align} whence,  applying Young's inequality   \begin{align}    \begin{split}    \frac{1}{2} \frac{d}{dt}     \Vert D^2 \rho \Vert_{L^2}    &\les    e^{2\epsilon t}    +    (\epsilon+\Vert \nabla u\Vert_{L^\infty})    \Vert D^2 \rho\Vert_{L^2}       ,       \end{split}    \llabel{0 PY zeVN i3yb 2msNYY Wz G 2DC PoG 1Vb Bx e9oZ GcTU 3AZuEK bk p 6rN eTX 0DS Mc zd91 nbSV DKEkVa zI q NKU Qap NBP 5B 32Ey prwP FLvuPi wR P l1G TdQ BZE Aw 3d90 v8P5 CPAnX4 Yo 2 q7s yr5 BW8 Hc T7tM ioha BW9U4q rb u mEQ 6Xz MKR 2B REFX k3ZO MVMYSw 9S F 5ek q0m yNK Gn H0qi vlRA 18CbEz id O iuy ZZ6 kRo oJ kLQ0 Ewmz sKlld6 Kr K JmR xls 12K G2 bv8v EWRTDFBSERSDFGSDFHDSFADSFSDFSVBASWERTDFSGSDFGFDGASFASDFDSFHFDGHFH158}   \end{align}  for all $t\geq0$. Applying a Gronwall argument and using \eqref{EWRTDFBSERSDFGSDFHDSFADSFSDFSVBASWERTDFSGSDFGFDGASFASDFDSFHFDGHFH165}, which holds for $t_0>0$ sufficiently large depending on $\epsilon$, we get
  \begin{align}    \begin{split}    \Vert D^2 \rho(t) \Vert_{L^2}    \leq    C_{\epsilon}e^{C \epsilon t}    (\Vert D^{2}\rho\Vert_{L^2}(t_0)+1)    \comma t\geq t_0    .    \end{split}    \label{EWRTDFBSERSDFGSDFHDSFADSFSDFSVBASWERTDFSGSDFGFDGASFASDFDSFHFDGHFH171}   \end{align} On the other hand,  using Gronwall's argument on $[0,t_0]$ with \eqref{EWRTDFBSERSDFGSDFHDSFADSFSDFSVBASWERTDFSGSDFGFDGASFASDFDSFHFDGHFH159} for $T=t_0$, we get   \begin{align}    \begin{split}    \Vert D^2 \rho(t) \Vert_{L^2}    \les    \Vert D^{2}\rho\Vert_{L^2}(0)+1    \comma t\in[0,t_0]    ,    \end{split}    \label{EWRTDFBSERSDFGSDFHDSFADSFSDFSVBASWERTDFSGSDFGFDGASFASDFDSFHFDGHFH172}   \end{align} where the constant depends on $t_0$ and thus on $\epsilon$. Combining \eqref{EWRTDFBSERSDFGSDFHDSFADSFSDFSVBASWERTDFSGSDFGFDGASFASDFDSFHFDGHFH171} and \eqref{EWRTDFBSERSDFGSDFHDSFADSFSDFSVBASWERTDFSGSDFGFDGASFASDFDSFHFDGHFH172}, we finally obtain \eqref{EWRTDFBSERSDFGSDFHDSFADSFSDFSVBASWERTDFSGSDFGFDGASFASDFDSFHFDGHFH14} with $C \epsilon$ replacing $\epsilon$. \colb \par To justify the a~priori bounds above, we consider the sequence of solutions   \begin{align}   \begin{split}   & \unp_t - \Delta \unp + \un \cdot \grad \unp + \grad \Pnp = \thetanp e_2    \\&   \thetanp_t + \un \cdot \grad \thetanp = - \unp \cdot e_2   \\&   \grad \cdot \unp = 0   ,   \end{split}    \label{EWRTDFBSERSDFGSDFHDSFADSFSDFSVBASWERTDFSGSDFGFDGASFASDFDSFHFDGHFH71}   \end{align} with the boundary condition $\unp|_{\FGSDFHGFHDFGHDFGH\Omega}=0$ and with the initial data   \begin{equation}    (\unp(0),\thetanp(0))=(u_0,\rho_0-x_2)    \llabel{LxfJ wrIcU6 Hx p q6p Fy7 Oim mo dXYt Kt0V VH22OC Aj f deT BAP vPl oK QzLE OQlq dpzxJ6 JI z Ujn TqY sQ4 BD QPW6 784x NUfsk0 aM 7 8qz MuL 9Mr Ac uVVK Y55n M7WqnB 2R C pGZ vHh WUN g9 3F2e RT8U umC62V H3 Z dJX LMS cca 1m xoOO 6oOL OVzfpO BO X 5Ev KuL z5s EW 8a9y otqk cKbDJN Us l pYM JpJ jOW Uy 2U4Y VKH6 kVC1Vx 1u v ykO yDs zo5 bz d36q WH1k J7JtkEWRTDFBSERSDFGSDFHDSFADSFSDFSVBASWERTDFSGSDFGFDGASFASDFDSFHFDGHFH72}    ,   \end{equation} for $n\in{\mathbb N}_0$. For $n=0$, we define   \begin{align}   \begin{split}   & u^{(0)}_t - \Delta u^{(0)} + \grad P^{(0)} = \theta^{(0)} e_2    \\&   \theta^{(0)}_t = - u^{(0)} \cdot e_2   \\&   \grad \cdot u^{(0)} = 0   ,   \end{split}    \llabel{g V1 J xqr Fnq mcU yZ JTp9 oFIc FAk0IT A9 3 SrL axO 9oU Z3 jG6f BRL1 iZ7ZE6 zj 8 G3M Hu8 6Ay jt 3flY cmTk jiTSYv CF t JLq cJP tN7 E3 POqG OKe0 3K3WV0 ep W XDQ C97 YSb AD ZUNp 81GF fCPbj3 iq E t0E NXy pLv fo Iz6z oFoF 9lkIun Xj Y yYL 52U bRB jx kQUS U9mm XtzIHO Cz 1 KH4 9ez 6Pz qW F223 C0Iz 3CsvuT R9 s VtQ CcM 1eo pD Py2l EEzL U0USJt Jb 9 zgyEWRTDFBSERSDFGSDFHDSFADSFSDFSVBASWERTDFSGSDFGFDGASFASDFDSFHFDGHFH73}   \end{align} with the boundary condition $u^{(0)}|_{\FGSDFHGFHDFGHDFGH\Omega}=0$ and with the initial data   \begin{equation}    (u^{(0)}(0),\theta^{(0)}(0))=(u_0,\rho_0-x_2)    \llabel{ Gyf iQ4 fo Cx26 k4jL E0ula6 aS I rZQ HER 5HV CE BL55 WCtB 2LCmve TD z Vcp 7UR gI7 Qu FbFw 9VTx JwGrzs VW M 9sM JeJ Nd2 VG GFsi WuqC 3YxXoJ GK w Io7 1fg sGm 0P YFBz X8eX 7pf9GJ b1 o XUs 1q0 6KP Ls MucN ytQb L0Z0Qq m1 l SPj 9MT etk L6 KfsC 6Zob Yhc2qu Xy 9 GPm ZYj 1Go ei feJ3 pRAf n6Ypy6 jN s 4Y5 nSE pqN 4m Rmam AGfY HhSaBr Ls D THC SEl UyR MEWRTDFBSERSDFGSDFHDSFADSFSDFSVBASWERTDFSGSDFGFDGASFASDFDSFHFDGHFH74}   .   \end{equation} Since the system \eqref{EWRTDFBSERSDFGSDFHDSFADSFSDFSVBASWERTDFSGSDFGFDGASFASDFDSFHFDGHFH71} is linear in $(\unp,\thetanp)$, it is easy to construct a local solution $(\unp,\thetanp)$. Also, our a~priori estimates apply to the sequence and  one may pass uniform bounds to the limit. Since the arguments are standard, we omit further details. \end{proof} \par \startnewsection{Interior bounds}{sec04} In this section, we establish the final result on the interior regularity of the second order derivatives. \par \begin{proof}[Proof of Theorem~\ref{T04}] In the proof, we work in the interior of the domain and thus localize the vorticity equation using a smooth cut-off function. With $\Omega'$ as in the statement,  consider a smooth function  $\eta \colon \mathbb{R}^2 \times [0,\infty) \to [0,1]$ such that $\supp \eta \subseteq \Omega\times [t_0/2,\infty)$ with $\eta = 1$ on $\Omega'' \times [3t_0/4,\infty)$,   where $\Omega''$ is an open set such that $\Omega' \Subset \Omega''\Subset \Omega$. In order to prove \eqref{EWRTDFBSERSDFGSDFHDSFADSFSDFSVBASWERTDFSGSDFGFDGASFASDFDSFHFDGHFH15}, we first claim that the vorticity $\omega=\curl u$ satisfies   \begin{align}   \Vert \grad \omega \Vert_{L^p([t_0,T]: L^p(\Omega'))}   \les  T^{1/p} + 1   ,   \label{EWRTDFBSERSDFGSDFHDSFADSFSDFSVBASWERTDFSGSDFGFDGASFASDFDSFHFDGHFH75}   \end{align} where the constant depends on $t_0$, and $\dist(\Omega',\FGSDFHGFHDFGHDFGH\Omega)$. Since \eqref{EWRTDFBSERSDFGSDFHDSFADSFSDFSVBASWERTDFSGSDFGFDGASFASDFDSFHFDGHFH15} and \eqref{EWRTDFBSERSDFGSDFHDSFADSFSDFSVBASWERTDFSGSDFGFDGASFASDFDSFHFDGHFH16} for $p=2$ follow from \eqref{EWRTDFBSERSDFGSDFHDSFADSFSDFSVBASWERTDFSGSDFGFDGASFASDFDSFHFDGHFH08}, we fix $p>2$.  We allow all constants to depend on~$p$ and $t_0$, where $t_0>0$ should be considered small. \par As in \cite{KW2}, we introduce the operator   \begin{align}   R = \FGSDFHGFHDFGHDFGH_1 (I-\Laplace)^{-1}   \llabel{h 66XU 7hNz pZVC5V nV 7 VjL 7kv WKf 7P 5hj6 t1vu gkLGdN X8 b gOX HWm 6W4 YE mxFG 4WaN EbGKsv 0p 4 OG0 Nrd uTe Za xNXq V4Bp mOdXIq 9a b PeD PbU Z4N Xt ohbY egCf xBNttE wc D YSD 637 jJ2 ms 6Ta1 J2xZ PtKnPw AX A tJA Rc8 n5d 93 TZi7 q6Wo nEDLwW Sz e Sue YFX 8cM hm Y6is 15pX aOYBbV fS C haL kBR Ks6 UO qG4j DVab fbdtny fi D BFEWRTDFBSERSDFGSDFHDSFADSFSDFSVBASWERTDFSGSDFGFDGASFASDFDSFHFDGHFH76}   \end{align} and a change of variable   \begin{align}   \zeta = \omega \eta - R(\rho \eta)   .   \label{EWRTDFBSERSDFGSDFHDSFADSFSDFSVBASWERTDFSGSDFGFDGASFASDFDSFHFDGHFH77}   \end{align} We shall apply $R$ to functions which are compactly supported in $\Omega$,  and we consider such functions extended to $\mathbb{R}^2$ by setting them identically to zero on $\Omega^{\text{c}}$. Recalling the vorticity formulation for \eqref{EWRTDFBSERSDFGSDFHDSFADSFSDFSVBASWERTDFSGSDFGFDGASFASDFDSFHFDGHFH01}, \begin{align}   \omega_t - \Laplace \omega + u \cdot \grad \omega = \FGSDFHGFHDFGHDFGH_1 \rho   ,   \llabel{I 7uh B39 FJ 6mYr CUUT f2X38J 43 K yZg 87i gFR 5R z1t3 jH9x lOg1h7 P7 W w8w jMJ qH3 l5 J5wU 8eH0 OogRCv L7 f JJg 1ug RfM XI GSuE Efbh 3hdNY3 x1 9 7jR qeP cdu sb fkuJ hEpw MvNBZV zL u qxJ 9b1 BTf Yk RJLj Oo1a EPIXvZ Aj v Xne fhK GsJ Ga wqjt U7r6 MPoydE H2 6 203 mGi JhF nT NCDB YlnP oKO6Pu XU 3 uu9 mSg 41v ma kk0E WUpS UtGBtD e6 d Kdx ZNT FuT EWRTDFBSERSDFGSDFHDSFADSFSDFSVBASWERTDFSGSDFGFDGASFASDFDSFHFDGHFH78}   \end{align} we have, as in \cite{KW2}, that   \begin{align}   \begin{split}   \zeta_t - \Laplace \zeta + u \cdot \grad \zeta    &=     [R, u \cdot \grad](\rho \eta) - N(\rho \eta)    -    \rho \FGSDFHGFHDFGHDFGH_1 \eta    -    2 \FGSDFHGFHDFGHDFGH_{j}(\omega \FGSDFHGFHDFGHDFGH_{j} \eta)    \\& \indeq    +    \omega(\eta_t + \Laplace \eta + u \cdot \grad \eta)    -    R( \rho(u \cdot \grad \eta))   ,   \end{split}   \label{EWRTDFBSERSDFGSDFHDSFADSFSDFSVBASWERTDFSGSDFGFDGASFASDFDSFHFDGHFH79} \end{align} where   \begin{align}   \llabel{i1 fMcM hq7P Ovf0hg Hl 8 fqv I3R K39 fn 9MaC Zgow 6e1iXj KC 5 lHO lpG pkK Xd Dxtz 0HxE fSMjXY L8 F vh7 dmJ kE8 QA KDo1 FqML HOZ2iL 9i I m3L Kva YiN K9 sb48 NxwY NR0nx2 t5 b WCk x2a 31k a8 fUIa RGzr 7oigRX 5s m 9PQ 7Sr 5St ZE Ymp8 VIWS hdzgDI 9v R F5J 81x 33n Ne fjBT VvGP vGsxQh Al G Fbe 1bQ i6J ap OJJa ceGq 1vvb8r F2 F 3M6 8eD lzG tX tVm5 y1EWRTDFBSERSDFGSDFHDSFADSFSDFSVBASWERTDFSGSDFGFDGASFASDFDSFHFDGHFH80}   N = ((I-\Laplace)^{-1}\Laplace - I)\FGSDFHGFHDFGHDFGH_1   ,   \end{align} which has the property that $\grad N$ is in the Calder\'on-Zygmund class.  The equation \eqref{EWRTDFBSERSDFGSDFHDSFADSFSDFSVBASWERTDFSGSDFGFDGASFASDFDSFHFDGHFH79} is obtained by a direct computation from   \begin{align}   \begin{split}    (\omega \eta)_t - \Delta (\omega \eta) + u\cdot \nabla (\omega \eta)    = \omega\eta_t      + \omega \Delta \eta      - 2 \FGSDFHGFHDFGHDFGH_{j}( \omega\FGSDFHGFHDFGHDFGH_{j} \eta)      + \omega u\cdot \nabla\eta      + \FGSDFHGFHDFGHDFGH_{1}(\rho \eta)      - \rho \FGSDFHGFHDFGHDFGH_{1} \eta   \end{split}   \llabel{4v mwIXa2 OG Y hxU sXJ 0qg l5 ZGAt HPZd oDWrSb BS u NKi 6KW gr3 9s 9tc7 WM4A ws1PzI 5c C O7Z 8y9 lMT LA dwhz Mxz9 hjlWHj bJ 5 CqM jht y9l Mn 4rc7 6Amk KJimvH 9r O tbc tCK rsi B0 4cFV Dl1g cvfWh6 5n x y9Z S4W Pyo QB yr3v fBkj TZKtEZ 7r U fdM icd yCV qn D036 HJWM tYfL9f yX x O7m IcF E1O uL QsAQ NfWv 6kV8Im 7Q 6 GsX NCV 0YP oC jnWn 6L25 qUMTe7 EWRTDFBSERSDFGSDFHDSFADSFSDFSVBASWERTDFSGSDFGFDGASFASDFDSFHFDGHFH69}   \end{align} and   \begin{equation}    (\rho\eta)_t    + u\cdot \nabla (\rho \eta)    = \rho u\cdot \nabla \eta    \llabel{1v a hnH DAo XAb Tc zhPc fjrj W5M5G0 nz N M5T nlJ WOP Lh M6U2 ZFxw pg4Nej P8 U Q09 JX9 n7S kE WixE Rwgy Fvttzp 4A s v5F Tnn MzL Vh FUn5 6tFY CxZ1Bz Q3 E TfD lCa d7V fo MwPm ngrD HPfZV0 aY k Ojr ZUw 799 et oYuB MIC4 ovEY8D OL N URV Q5l ti1 iS NZAd wWr6 Q8oPFf ae 5 lAR 9gD RSi HO eJOW wxLv 20GoMt 2H z 7Yc aly PZx eR uFM0 7gaV 9UIz7S 43 k 5Tr ZEWRTDFBSERSDFGSDFHDSFADSFSDFSVBASWERTDFSGSDFGFDGASFASDFDSFHFDGHFH70}   \end{equation} and then using the identity $\FGSDFHGFHDFGHDFGH_{1} + \Delta R = - N$. Note that both operators $R$ and $N$ commute with translations (and hence derivatives) and they are smoothing of order one, i.e., they satisfy   \begin{equation}    \Vert R f\Vert_{W^{1,p}},    \Vert N f\Vert_{W^{1,p}}    \les \Vert f\Vert_{L^p}    \comma   f\in L^{p}(\mathbb{R}^{2})    ,       \label{EWRTDFBSERSDFGSDFHDSFADSFSDFSVBASWERTDFSGSDFGFDGASFASDFDSFHFDGHFH64}   \end{equation} for $ p\in(1,\infty)$, where the constant depends on~$p$; the property \eqref{EWRTDFBSERSDFGSDFHDSFADSFSDFSVBASWERTDFSGSDFGFDGASFASDFDSFHFDGHFH64} can be verified by computing the Fourier multiplier symbols corresponding to $R$ and $N$ (or cf.~\cite{KW2}). Since $u$ is divergence free, we may rewrite   \begin{align}   [R, u_j \FGSDFHGFHDFGHDFGH_j](\rho \eta)   =   R(u_j \FGSDFHGFHDFGHDFGH_j (\rho \eta)) - u_j \FGSDFHGFHDFGHDFGH_j R(\rho \eta)   =   \FGSDFHGFHDFGHDFGH_j R (u_j \rho \eta) - u_j \FGSDFHGFHDFGHDFGH_j R(\rho \eta)   .   \llabel{iD Mt7 pE NCYi uHL7 gac7Gq yN 6 Z1u x56 YZh 2d yJVx 9MeU OMWBQf l0 E mIc 5Zr yfy 3i rahC y9Pi MJ7ofo Op d enn sLi xZx Jt CjC9 M71v O0fxiR 51 m FIB QRo 1oW Iq 3gDP stD2 ntfoX7 YU o S5k GuV IGM cf HZe3 7ZoG A1dDmk XO 2 KYR LpJ jII om M6Nu u8O0 jO5Nab Ub R nZn 15k hG9 4S 21V4 Ip45 7ooaiP u2 j hIz osW FDu O5 HdGr djvv tTLBjo vL L iCo 6L5 Lwa Pm EWRTDFBSERSDFGSDFHDSFADSFSDFSVBASWERTDFSGSDFGFDGASFASDFDSFHFDGHFH81}   \end{align} \par To acquire $L^p$ space-time estimates, we rewrite our solution as $\zeta = \zeta^{(1)} + \zeta^{(2)}$, where $\zeta^{(1)}$ satisfies   \begin{align}   \begin{split}   &\zeta_t^{(1)} - \Laplace \zeta^{(1)} = f   \\&   \zeta^{(1)} \bigl{|}_{t=0} =0   \end{split}   \llabel{vD6Z pal6 9Ljn11 re T 2CP mvj rL3 xH mDYK uv5T npC1fM oU R RTo Loi lk0 FE ghak m5M9 cOIPdQ lG D LnX erC ykJ C1 0FHh vvnY aTGuqU rf T QPv wEq iHO vO hD6A nXuv GlzVAv pz d Ok3 6ym yUo Fb AcAA BItO es52Vq d0 Y c7U 2gB t0W fF VQZh rJHr lBLdCx 8I o dWp AlD S8C HB rNLz xWp6 ypjuwW mg X toy 1vP bra uH yMNb kUrZ D6Ee2f zI D tkZ Eti Lmg re 1woD juLB EWRTDFBSERSDFGSDFHDSFADSFSDFSVBASWERTDFSGSDFGFDGASFASDFDSFHFDGHFH83}   \end{align} with \begin{align}   f   =   \omega(\eta_t + \Laplace \eta + u \cdot \grad \eta)   -   R( \rho(u \cdot \grad \eta))   -   N(\rho \eta)   -   u \cdot \grad R(\rho \eta)   -   \rho \FGSDFHGFHDFGHDFGH_1 \eta   ,   \llabel{BSdasY Vc F Uhy ViC xB1 5y Ltql qoUh gL3bZN YV k orz wa3 650 qW hF22 epiX cAjA4Z V4 b cXx uB3 NQN p0 GxW2 Vs1z jtqe2p LE B iS3 0E0 NKH gY N50v XaK6 pNpwdB X2 Y v7V 0Ud dTc Pi dRNN CLG4 7Fc3PL Bx K 3Be x1X zyX cj 0Z6a Jk0H KuQnwd Dh P Q1Q rwA 05v 9c 3pnz ttzt x2IirW CZ B oS5 xlO KCi D3 WFh4 dvCL QANAQJ Gg y vOD NTD FKj Mc 0RJP m4HU SQkLnT Q4 EWRTDFBSERSDFGSDFHDSFADSFSDFSVBASWERTDFSGSDFGFDGASFASDFDSFHFDGHFH84}   \end{align} while for $\zeta^{(2)}$ we have   \begin{align}   \begin{split}   &\zeta_t^{(2)} - \Laplace \zeta^{(2)} = \grad \cdot g   \\&   \zeta^{(2)} \bigl{|}_{t=0} =0   ,   \end{split}   \llabel{Y 6CC MvN jAR Zb lir7 RFsI NzHiJl cg f xSC Hts ZOG 1V uOzk 5G1C LtmRYI eD 3 5BB uxZ JdY LO CwS9 lokS NasDLj 5h 8 yni u7h u3c di zYh1 PdwE l3m8Xt yX Q RCA bwe aLi N8 qA9N 6DRE wy6gZe xs A 4fG EKH KQP PP KMbk sY1j M4h3Jj gS U One p1w RqN GA grL4 c18W v4kchD gR x 7Gj jIB zcK QV f7gA TrZx Oy6FF7 y9 3 iuu AQt 9TK Rx S5GO TFGx 4Xx1U3 R4 s 7U1 mpa EWRTDFBSERSDFGSDFHDSFADSFSDFSVBASWERTDFSGSDFGFDGASFASDFDSFHFDGHFH85}   \end{align} where    \begin{align}   g   =   - u \zeta    - 2 \omega \grad \eta    + R( u \rho \eta)   .   \llabel{bpD Hg kicx aCjk hnobr0 p4 c ody xTC kVj 8t W4iP 2OhT RF6kU2 k2 o oZJ Fsq Y4B FS NI3u W2fj OMFf7x Jv e ilb UVT ArC Tv qWLi vbRp g2wpAJ On l RUE PKh j9h dG M0Mi gcqQ wkyunB Jr T LDc Pgn OSC HO sSgQ sR35 MB7Bgk Pk 6 nJh 01P Cxd Ds w514 O648 VD8iJ5 4F W 6rs 6Sy qGz MK fXop oe4e o52UNB 4Q 8 f8N Uz8 u2n GO AXHW gKtG AtGGJs bm z 2qj vSv GBu 5e 4JgEWRTDFBSERSDFGSDFHDSFADSFSDFSVBASWERTDFSGSDFGFDGASFASDFDSFHFDGHFH86}   \end{align} Using the $L^{p}W^{2,p}$ regularity for the nonhomogeneous heat equation and the Gagliardo-Nirenberg inequality, we have   \begin{align}    \begin{split}    \Vert D\zeta^{(1)}\Vert_{L^{p}L^{p}(\mathbb{R}^2\times(0,\infty))}    \les    \Vert D^2\zeta^{(1)}\Vert_{L^{p}L^{2p/(p+2)}(\mathbb{R}^2\times(0,\infty))}    \les    \Vert f\Vert_{L^{p}L^{2p/(p+2)}(\mathbb{R}^2\times(0,\infty))}     ;    \end{split}    \label{EWRTDFBSERSDFGSDFHDSFADSFSDFSVBASWERTDFSGSDFGFDGASFASDFDSFHFDGHFH27}   \end{align} observe that $2p/(p+2)>1$ since $p>2$. Similarly, using the $L^{p}W^{1,p}$ regularity for the nonhomogeneous heat equation in divergence form, we have   \begin{align}    \begin{split}    \Vert D\zeta^{(2)}\Vert_{L^{p}L^{p}(\mathbb{R}^2\times(0,\infty))}    \les
   \Vert g\Vert_{L^{p}L^{p}(\mathbb{R}^2\times(0,\infty))}     .    \end{split}    \label{EWRTDFBSERSDFGSDFHDSFADSFSDFSVBASWERTDFSGSDFGFDGASFASDFDSFHFDGHFH49}   \end{align} For the right-hand side of \eqref{EWRTDFBSERSDFGSDFHDSFADSFSDFSVBASWERTDFSGSDFGFDGASFASDFDSFHFDGHFH27}, we use \eqref{EWRTDFBSERSDFGSDFHDSFADSFSDFSVBASWERTDFSGSDFGFDGASFASDFDSFHFDGHFH64} to obtain    \begin{align}   \begin{split}   \Vert f \Vert_{L^{2p/(p+2)}}   &\les   \Vert \omega \Vert_{L^{2}}   (   \Vert \eta_t \Vert_{L^{p}}   + \Vert \Laplace \eta \Vert_{L^{p}}   + \Vert u \Vert_{L^\infty} \Vert \grad \eta \Vert_{L^{p}}   )   +   \Vert \rho \Vert_{L^{2}} \Vert u \Vert_{L^\infty} \Vert \grad \eta \Vert_{L^{p}}   \\& \indeq   +   \Vert \rho \Vert_{L^{2}} \Vert \eta \Vert_{L^{p}}   +   \Vert u \Vert_{L^\infty} \Vert \rho \Vert_{L^{2}} \Vert \eta \Vert_{L^{p}}   +   \Vert \rho \Vert_{L^{2}} \Vert \FGSDFHGFHDFGHDFGH_1 \eta \Vert_{L^{p}}     \\&   \les   \Vert \omega \Vert_{L^{2}}   +   1   \les 1   ,   \end{split}   \label{EWRTDFBSERSDFGSDFHDSFADSFSDFSVBASWERTDFSGSDFGFDGASFASDFDSFHFDGHFH88}   \end{align} for every $t\geq 0$, where the domains are understood to be $\mathbb{R}^2$. For the right-hand side in \eqref{EWRTDFBSERSDFGSDFHDSFADSFSDFSVBASWERTDFSGSDFGFDGASFASDFDSFHFDGHFH49}, we determine that   \begin{align}   \begin{split}   \Vert g \Vert_{L^p}   &\les   \Vert u\Vert_{L^{2p}}   \Vert \zeta \Vert_{L^{2p}}   +   \Vert \omega \Vert_{L^{2p}}   \Vert \grad \eta \Vert_{L^{2p}}   +   \Vert u \Vert_{L^\infty}    \Vert \rho \Vert_{L^{p}}    \Vert \eta \Vert_{L^{\infty}}   \\&   \les    \Vert u\Vert_{L^{2p}}   \Vert \zeta \Vert_{L^{2p}}   +   \Vert \omega \Vert_{L^{2p}}   +   1   \end{split}   \label{EWRTDFBSERSDFGSDFHDSFADSFSDFSVBASWERTDFSGSDFGFDGASFASDFDSFHFDGHFH90}   \end{align} for every $t\geq 0$, by \eqref{EWRTDFBSERSDFGSDFHDSFADSFSDFSVBASWERTDFSGSDFGFDGASFASDFDSFHFDGHFH64}. To bound the  right-hand side of \eqref{EWRTDFBSERSDFGSDFHDSFADSFSDFSVBASWERTDFSGSDFGFDGASFASDFDSFHFDGHFH90}, we write   \begin{align}   \Vert \zeta \Vert_{L^{q}} \les \Vert \omega\eta \Vert_{L^q}    + \Vert R(\rho\eta)\Vert_{L^q}    \les   \Vert \omega\eta \Vert_{L^q}    + \Vert \rho\Vert_{L^2}    \les   1    \comma q\in[2,\infty)   .   \label{EWRTDFBSERSDFGSDFHDSFADSFSDFSVBASWERTDFSGSDFGFDGASFASDFDSFHFDGHFH92}   \end{align} Therefore, we obtain $\Vert g\Vert_{L^p}\les 1$ for all $t\geq 0$. This fact and \eqref{EWRTDFBSERSDFGSDFHDSFADSFSDFSVBASWERTDFSGSDFGFDGASFASDFDSFHFDGHFH88} imply  by integration that the left-hand sides of  \eqref{EWRTDFBSERSDFGSDFHDSFADSFSDFSVBASWERTDFSGSDFGFDGASFASDFDSFHFDGHFH27} and \eqref{EWRTDFBSERSDFGSDFHDSFADSFSDFSVBASWERTDFSGSDFGFDGASFASDFDSFHFDGHFH49} are bounded by $T^{1/p}$ for $T\geq t_0$, from where   \begin{equation}    \Vert D\zeta\Vert_{L^{p}L^{p}(\mathbb{R}^2\times(0,\infty))}    \les    T^{1/p}    \label{EWRTDFBSERSDFGSDFHDSFADSFSDFSVBASWERTDFSGSDFGFDGASFASDFDSFHFDGHFH65}   \end{equation} and thus   \begin{align}   \Vert \grad (\omega \eta) \Vert_{L^{p}L^{p}(\mathbb{R}^2\times(0,\infty))}   \les   \Vert \grad \zeta \Vert_{L^{p}L^{p}(\mathbb{R}^2\times(0,\infty))}    + \Vert R\nabla(\rho \eta) \Vert_{L^pL^{p}(\mathbb{R}^2\times(0,\infty))}   \les   T^{1/p} + 1   ,   \llabel{L Aqrm gMmS08 ZF s xQm 28M 3z4 Ho 1xxj j8Uk bMbm8M 0c L PL5 TS2 kIQ jZ Kb9Q Ux2U i5Aflw 1S L DGI uWU dCP jy wVVM 2ct8 cmgOBS 7d Q ViX R8F bta 1m tEFj TO0k owcK2d 6M Z iW8 PrK PI1 sX WJNB cREV Y4H5QQ GH b plP bwd Txp OI 5OQZ AKyi ix7Qey YI 9 1Ea 16r KXK L2 ifQX QPdP NL6EJi Hc K rBs 2qG tQb aq edOj Lixj GiNWr1 Pb Y SZe Sxx Fin aK 9Eki CHV2 a13EWRTDFBSERSDFGSDFHDSFADSFSDFSVBASWERTDFSGSDFGFDGASFASDFDSFHFDGHFH93}   \end{align} which proves \eqref{EWRTDFBSERSDFGSDFHDSFADSFSDFSVBASWERTDFSGSDFGFDGASFASDFDSFHFDGHFH75}. The bound \eqref{EWRTDFBSERSDFGSDFHDSFADSFSDFSVBASWERTDFSGSDFGFDGASFASDFDSFHFDGHFH15} then follows by a simple application of the interior elliptic estimate connecting $u$ and $\omega$. \par The pointwise in time bound in \eqref{EWRTDFBSERSDFGSDFHDSFADSFSDFSVBASWERTDFSGSDFGFDGASFASDFDSFHFDGHFH16} follows once we obtain   \begin{align}   \begin{split}   \Vert \grad \omega(t) \Vert_{L^p(\Omega')}    \les    t^{1/4 + 2/p + 1/p^2}    \comma t\geq t_0     ,     \end{split}   \label{EWRTDFBSERSDFGSDFHDSFADSFSDFSVBASWERTDFSGSDFGFDGASFASDFDSFHFDGHFH94}   \end{align} where the constant depends on $t_0$, $p$, and $\dist(\Omega',\FGSDFHGFHDFGHDFGH\Omega)$. To prove \eqref{EWRTDFBSERSDFGSDFHDSFADSFSDFSVBASWERTDFSGSDFGFDGASFASDFDSFHFDGHFH94}, we begin by introducing a second smooth cut-off function $\phi \colon \mathbb{R}^2 \times [0,\infty) \to [0,1]$ for which   \begin{equation}    \supp \phi \subseteq \{\eta=1\}=\{(x,t) \in \mathbb{R}^2 \times [0,\infty) : \eta(x,t) = 1 \}       \llabel{f7G 3G 3 oDK K0i bKV y4 53E2 nFQS 8Hnqg0 E3 2 ADd dEV nmJ 7H Bc1t 2K2i hCzZuy 9k p sHn 8Ko uAR kv sHKP y8Yo dOOqBi hF 1 Z3C vUF hmj gB muZq 7ggW Lg5dQB 1k p Fxk k35 GFo dk 00YD 13qI qqbLwy QC c yZR wHA fp7 9o imtC c5CV 8cEuwU w7 k 8Q7 nCq WkM gY rtVR IySM tZUGCH XV 9 mr9 GHZ ol0 VE eIjQ vwgw 17pDhX JS F UcY bqU gnG V8 IFWb S1GX az0ZTt 81 w 7EWRTDFBSERSDFGSDFHDSFADSFSDFSVBASWERTDFSGSDFGFDGASFASDFDSFHFDGHFH95}   \end{equation} and is such that $\phi=1$ on $\Omega'\times[t_0,\infty)$. Denote   \begin{equation}    \tzeta    = \zeta \phi    .    \llabel{En IhF F72 v2 PkWO Xlkr w6IPu5 67 9 vcW 1f6 z99 lM 2LI1 Y6Na axfl18 gT 0 gDp tVl CN4 jf GSbC ro5D v78Cxa uk Y iUI WWy YDR w8 z7Kj Px7C hC7zJv b1 b 0rF d7n Mxk 09 1wHv y4u5 vLLsJ8 Nm A kWt xuf 4P5 Nw P23b 06sF NQ6xgD hu R GbK 7j2 O4g y4 p4BL top3 h2kfyI 9w O 4Aa EWb 36Y yH YiI1 S3CO J7aN1r 0s Q OrC AC4 vL7 yr CGkI RlNu GbOuuk 1a w LDK 2zl Ka4EWRTDFBSERSDFGSDFHDSFADSFSDFSVBASWERTDFSGSDFGFDGASFASDFDSFHFDGHFH96}   \end{equation} Using \eqref{EWRTDFBSERSDFGSDFHDSFADSFSDFSVBASWERTDFSGSDFGFDGASFASDFDSFHFDGHFH79}, we find that   \begin{align}   \begin{split}   &   \tzeta_t - \Laplace \tzeta + u \cdot \grad \tzeta   \\&\indeq   =    \bigl([R,u\cdot \grad](\rho \eta) - N(\rho \eta)    \bigr)\phi   - R( \rho(u \cdot \grad \eta)) \phi   - 2 \grad \zeta \cdot \grad \phi   + \zeta (\phi_t - \Laplace \phi + u\cdot \grad \phi)   ;   \end{split}   \label{EWRTDFBSERSDFGSDFHDSFADSFSDFSVBASWERTDFSGSDFGFDGASFASDFDSFHFDGHFH97}   \end{align} note that the terms in \eqref{EWRTDFBSERSDFGSDFHDSFADSFSDFSVBASWERTDFSGSDFGFDGASFASDFDSFHFDGHFH79} containing derivatives of $\eta$ vanish after multiplication with $\phi$, except for the term involving $R$, which is a non-local operator. The main reason for introducing the second cut-off function $\phi$ is that  $\zeta$ does not vanish on the boundary $\FGSDFHGFHDFGHDFGH\Omega$ due to nonlocality of $R$; cf.~the definition~\eqref{EWRTDFBSERSDFGSDFHDSFADSFSDFSVBASWERTDFSGSDFGFDGASFASDFDSFHFDGHFH77}. In order to estimate $\nabla \tzeta$, we apply  $\FGSDFHGFHDFGHDFGH_k$ to \eqref{EWRTDFBSERSDFGSDFHDSFADSFSDFSVBASWERTDFSGSDFGFDGASFASDFDSFHFDGHFH97} for $k=1$, $2$, multiply by $|\FGSDFHGFHDFGHDFGH_k\tzeta|^{2p-2}\FGSDFHGFHDFGHDFGH_k\tzeta$, integrate, and sum in $k$ to acquire   \begin{align}   \begin{split}   &\frac{1}{2p}\frac{d}{dt} \sum_k \Vert \FGSDFHGFHDFGHDFGH_k \tzeta \Vert_{L^{2p}}^{2p}   -   \sum_k \int \Laplace \FGSDFHGFHDFGHDFGH_k\tzeta |\FGSDFHGFHDFGHDFGH_k\tzeta|^{2p-2}\FGSDFHGFHDFGHDFGH_k\tzeta   \\&\indeq   =    - \sum_k \int \FGSDFHGFHDFGHDFGH_k(u_j \FGSDFHGFHDFGHDFGH_j \tzeta)   |\FGSDFHGFHDFGHDFGH_k\tzeta|^{2p-2}\FGSDFHGFHDFGHDFGH_k\tzeta   +   \sum_k \int \FGSDFHGFHDFGHDFGH_k (\phi[R,u\cdot \grad](\rho \eta))   |\FGSDFHGFHDFGHDFGH_k\tzeta|^{2p-2}\FGSDFHGFHDFGHDFGH_k\tzeta   \\& \indeq\indeq   -    \sum_k \int \FGSDFHGFHDFGHDFGH_k (\phi N(\rho \eta) )    |\FGSDFHGFHDFGHDFGH_k\tzeta|^{2p-2}\FGSDFHGFHDFGHDFGH_k\tzeta   -   \sum_k \int \FGSDFHGFHDFGHDFGH_k   (\phi R( \rho(u \cdot \grad \eta)))   |\FGSDFHGFHDFGHDFGH_k\tzeta|^{2p-2}\FGSDFHGFHDFGHDFGH_k\tzeta   \\& \indeq\indeq   -2   \sum_k \int \FGSDFHGFHDFGHDFGH_k(\FGSDFHGFHDFGHDFGH_j \zeta \FGSDFHGFHDFGHDFGH_j \phi)|\FGSDFHGFHDFGHDFGH_k\tzeta|^{2p-2}\FGSDFHGFHDFGHDFGH_k\tzeta   \\& \indeq\indeq   +   \sum_k \int \FGSDFHGFHDFGHDFGH_k(\zeta (\phi_t - \Laplace \phi + u\cdot \grad \phi))   |\FGSDFHGFHDFGHDFGH_k\tzeta|^{2p-2}\FGSDFHGFHDFGHDFGH_k\tzeta   .   \end{split}   \label{EWRTDFBSERSDFGSDFHDSFADSFSDFSVBASWERTDFSGSDFGFDGASFASDFDSFHFDGHFH98}   \end{align} The second term on the left-hand side of \eqref{EWRTDFBSERSDFGSDFHDSFADSFSDFSVBASWERTDFSGSDFGFDGASFASDFDSFHFDGHFH98} is estimated as   \begin{align}   \begin{split}    &-\sum_k \int \Laplace \FGSDFHGFHDFGHDFGH_k\tzeta |\FGSDFHGFHDFGHDFGH_k\tzeta|^{2p-2}\FGSDFHGFHDFGHDFGH_k\tzeta    =    \frac{2p-1}{p^2} \sum_k \int \FGSDFHGFHDFGHDFGH_j (|\FGSDFHGFHDFGHDFGH_k\tzeta|^p) \FGSDFHGFHDFGHDFGH_j (|\FGSDFHGFHDFGHDFGH_k\tzeta|^p)    \geq \frac{1}{p}\sum_k \Vert \nabla(|\FGSDFHGFHDFGHDFGH_k\tzeta|^p) \Vert_{L^2}^2    =     \frac{1}{p}    \bar D    ,   \end{split}   \label{EWRTDFBSERSDFGSDFHDSFADSFSDFSVBASWERTDFSGSDFGFDGASFASDFDSFHFDGHFH99}   \end{align}  where we denoted $\bar D= \sum_k \Vert \nabla(|\FGSDFHGFHDFGHDFGH_k\tzeta|^p) \Vert_{L^2}^2$. For the first term on the right-hand side of \eqref{EWRTDFBSERSDFGSDFHDSFADSFSDFSVBASWERTDFSGSDFGFDGASFASDFDSFHFDGHFH98}, we use the incompressibility of $u$ to determine that   \begin{align}   \begin{split}    &-\sum_k \int \FGSDFHGFHDFGHDFGH_k(u_j \FGSDFHGFHDFGHDFGH_j \tzeta) |\FGSDFHGFHDFGHDFGH_k\tzeta|^{2p-2}\FGSDFHGFHDFGHDFGH_k\tzeta    =    -    \sum_k \int \FGSDFHGFHDFGHDFGH_k u_j \FGSDFHGFHDFGHDFGH_j \tzeta   |\FGSDFHGFHDFGHDFGH_k\tzeta|^{2p-2}\FGSDFHGFHDFGHDFGH_k\tzeta   \\& \indeq    \les   \Vert \grad u \Vert_{L^2} \Vert \grad \tzeta \Vert_{L^{4p}}   \sum_k \Vert |\FGSDFHGFHDFGHDFGH_k \tzeta|^{2p-1}\Vert_{L^{4p/(2p-1)}}   \les   o(1) \sum_k \Vert \FGSDFHGFHDFGHDFGH_k \tzeta \Vert_{L^{4p}}^{2p}   \colb    ,   \end{split}   \label{EWRTDFBSERSDFGSDFHDSFADSFSDFSVBASWERTDFSGSDFGFDGASFASDFDSFHFDGHFH100}   \end{align}  where $o(1)$ denotes a function which is bounded on $[0,\infty)$ and converges to~$0$ as $t\to\infty$. Applying the estimate   \begin{align}   \begin{split}   \Vert \FGSDFHGFHDFGHDFGH_k \tzeta \Vert_{L^{4p}}^{2p}    =    \Vert |\FGSDFHGFHDFGHDFGH_k \tzeta|^p \Vert_{L^{4}}^{2}   \les   \Vert |\FGSDFHGFHDFGHDFGH_k \tzeta|^p \Vert_{L^{2}}   \Vert \grad(|\FGSDFHGFHDFGHDFGH_k\tzeta|^p) \Vert_{L^2}   \les   {\bar D}^{1/2}\Vert \FGSDFHGFHDFGHDFGH_k \tzeta \Vert_{L^{2p}}^p   \end{split}   \llabel{ 0h yJnD V4iF xsqO00 1r q CeO AO2 es7 DR aCpU G54F 2i97xS Qr c bPZ 6K8 Kud n9 e6SY o396 Fr8LUx yX O jdF sMr l54 Eh T8vr xxF2 phKPbs zr l pMA ubE RMG QA aCBu 2Lqw Gasprf IZ O iKV Vbu Vae 6a bauf y9Kc Fk6cBl Z5 r KUj htW E1C nt 9Rmd whJR ySGVSO VT v 9FY 4uz yAH Sp 6yT9 s6R6 oOi3aq Zl L 7bI vWZ 18c Fa iwpt C1nd Fyp4oK xD f Qz2 813 6a8 zX wsGl YEWRTDFBSERSDFGSDFHDSFADSFSDFSVBASWERTDFSGSDFGFDGASFASDFDSFHFDGHFH101}   \end{align}  in \eqref{EWRTDFBSERSDFGSDFHDSFADSFSDFSVBASWERTDFSGSDFGFDGASFASDFDSFHFDGHFH100}, we obtain   \begin{align}   \begin{split}   &   -\sum_k \int \FGSDFHGFHDFGHDFGH_k(u_j \FGSDFHGFHDFGHDFGH_j \tzeta) |\FGSDFHGFHDFGHDFGH_k\tzeta|^{2p-2}\FGSDFHGFHDFGHDFGH_k\tzeta   \leq   o(1)   {\bar D}^{1/2}\sum_k \Vert \FGSDFHGFHDFGHDFGH_k \tzeta \Vert_{L^{2p}}^p   \leq   \frac{{\bar D}}{8} + o(1) \sum_k\Vert \FGSDFHGFHDFGHDFGH_k \tzeta \Vert_{L^{2p}}^{2p}   .
  \end{split}   \label{EWRTDFBSERSDFGSDFHDSFADSFSDFSVBASWERTDFSGSDFGFDGASFASDFDSFHFDGHFH102}   \end{align}  For the second term on the right-hand side of \eqref{EWRTDFBSERSDFGSDFHDSFADSFSDFSVBASWERTDFSGSDFGFDGASFASDFDSFHFDGHFH98}, we use integration by parts and write   \begin{align}   \begin{split}   &\sum_k \int \FGSDFHGFHDFGHDFGH_k (\phi [R,u\cdot \grad](\rho \eta) )   |\FGSDFHGFHDFGHDFGH_k\tzeta|^{2p-2}\FGSDFHGFHDFGHDFGH_k\tzeta   =   -(2p-1)\sum_k \int \phi [R,u\cdot \grad](\rho \eta) |\FGSDFHGFHDFGHDFGH_k\tzeta|^{2p-2}\FGSDFHGFHDFGHDFGH_{kk}\tzeta   \\& \indeq   =   -\frac{2p-1}{p} \sum_k \int \phi [R,u\cdot \grad](\rho \eta) |\FGSDFHGFHDFGHDFGH_k\tzeta|^{p-2}\FGSDFHGFHDFGHDFGH_k\tzeta \FGSDFHGFHDFGHDFGH_k (|\FGSDFHGFHDFGHDFGH_k\tzeta|^p)   \\& \indeq   \les   \Vert \phi [R,u\cdot \grad](\rho \eta) \Vert_{L^{2p}}   \sum_k \Vert |\FGSDFHGFHDFGHDFGH_k\tzeta|^{p-1} \Vert_{L^{2p/(p-1)}} \Vert \grad(|\FGSDFHGFHDFGHDFGH_k\tzeta|^p) \Vert_{L^2}   \\& \indeq   \les   {\bar D}^{1/2}\Vert \phi [R,u\cdot \grad](\rho \eta) \Vert_{L^{2p}}   \sum_k \Vert \FGSDFHGFHDFGHDFGH_k\tzeta \Vert_{L^{2p}}^{p-1}   .   \end{split}   \label{EWRTDFBSERSDFGSDFHDSFADSFSDFSVBASWERTDFSGSDFGFDGASFASDFDSFHFDGHFH103}   \end{align} For the second factor in the last expression, we have   \begin{align}   \begin{split}   \Vert \phi [R,u\cdot \grad](\rho \eta) \Vert_{L^{2p}}   &\les   \Vert \phi \Vert_{L^{\infty}} \Vert R(u_j\FGSDFHGFHDFGHDFGH_j (\rho\eta) ) - u_j\FGSDFHGFHDFGHDFGH_j R(\rho\eta) \Vert_{L^{2p}}   \\&   \les   \Vert \FGSDFHGFHDFGHDFGH_jR(u_j (\rho\eta)) \Vert_{L^{2p}} + \Vert u \Vert_{L^\infty}\Vert \FGSDFHGFHDFGHDFGH_j R(\rho\eta) \Vert_{L^{2p}}   \les   \Vert \rho\eta \Vert_{L^{2p}} \les 1   ,   \end{split}   \label{EWRTDFBSERSDFGSDFHDSFADSFSDFSVBASWERTDFSGSDFGFDGASFASDFDSFHFDGHFH104}   \end{align} using the incompressibility of $u$ and    \begin{align}   \Vert \rho(t) \Vert_{L^{2p}} \les 1   ,   \label{EWRTDFBSERSDFGSDFHDSFADSFSDFSVBASWERTDFSGSDFGFDGASFASDFDSFHFDGHFH89}   \end{align} which follows from $   \Vert \rho_0 \Vert_{L^{2p}} \les \Vert \rho_0\Vert_{H^{1}}\les 1 $ and the $L^{p}$ conservation for $\rho$. (Recall that all constants depend on~$p$.) Thus, by \eqref{EWRTDFBSERSDFGSDFHDSFADSFSDFSVBASWERTDFSGSDFGFDGASFASDFDSFHFDGHFH103}--\eqref{EWRTDFBSERSDFGSDFHDSFADSFSDFSVBASWERTDFSGSDFGFDGASFASDFDSFHFDGHFH104}, we have   \begin{align}   \begin{split}   &\sum_k \int \FGSDFHGFHDFGHDFGH_k (\phi [R,u\cdot \grad](\rho \eta) )   |\FGSDFHGFHDFGHDFGH_k\tzeta|^{2p-2}\FGSDFHGFHDFGHDFGH_k\tzeta   \leq   C{\bar D}^{1/2}\sum_k\Vert \FGSDFHGFHDFGHDFGH_k\tzeta \Vert_{L^{2p}}^{p-1}   \\& \indeq   \leq   \frac{{\bar D}}{8} + C \sum_k\Vert \FGSDFHGFHDFGHDFGH_k\tzeta \Vert_{L^{2p}}^{2p-2}   .   \end{split}   \label{EWRTDFBSERSDFGSDFHDSFADSFSDFSVBASWERTDFSGSDFGFDGASFASDFDSFHFDGHFH105}   \end{align} For the third term on the right-hand side of \eqref{EWRTDFBSERSDFGSDFHDSFADSFSDFSVBASWERTDFSGSDFGFDGASFASDFDSFHFDGHFH98}, we obtain   \begin{align}   \begin{split}   &-\sum_k \int \FGSDFHGFHDFGHDFGH_k (\phi N(\rho \eta) )   |\FGSDFHGFHDFGHDFGH_k\tzeta|^{2p-2}\FGSDFHGFHDFGHDFGH_k\tzeta    \les    \sum_k \Vert \FGSDFHGFHDFGHDFGH_k (\phi N(\rho \eta))\Vert_{L^{2p}}           \Vert |\FGSDFHGFHDFGHDFGH_k\tzeta|^{2p-1}\Vert_{L^{2p/(2p-1)}} \\&\indeq    \les    (    \Vert \grad \phi \Vert_{L^\infty} \Vert N(\rho \eta) \Vert_{L^{2p}}    +    \Vert \phi\Vert_{L^\infty} \Vert \grad N(\rho \eta) \Vert_{L^{2p}}     )    \sum_k    \Vert \FGSDFHGFHDFGHDFGH_k\tzeta\Vert_{L^{2p}}^{2p-1}    \les    \Vert \rho\eta\Vert_{L^{2p}}    \sum_k \Vert \FGSDFHGFHDFGHDFGH_k\tzeta\Vert_{L^{2p}}^{2p-1}    \\&\indeq    \les    \sum_k \Vert \FGSDFHGFHDFGHDFGH_k\tzeta\Vert_{L^{2p}}^{2p-1}    .   \end{split}   \llabel{sh9 Gp3Tal nr R UKt tBK eFr 45 43qU 2hh3 WbYw09 g2 W LIX zvQ zMk j5 f0xL seH9 dscinG wu P JLP 1gE N5W qY sSoW Peqj MimTyb Hj j cbn 0NO 5hz P9 W40r 2w77 TAoz70 N1 a u09 boc DSx Gc 3tvK LXaC 1dKgw9 H3 o 2kE oul In9 TS PyL2 HXO7 tSZse0 1Z 9 Hds lDq 0tm SO AVqt A1FQ zEMKSb ak z nw8 39w nH1 Dp CjGI k5X3 B6S6UI 7H I gAa f9E V33 Bk kuo3 FyEi 8Ty2ABEWRTDFBSERSDFGSDFHDSFADSFSDFSVBASWERTDFSGSDFGFDGASFASDFDSFHFDGHFH106}   \end{align} For the fourth term on the right-hand side of~\eqref{EWRTDFBSERSDFGSDFHDSFADSFSDFSVBASWERTDFSGSDFGFDGASFASDFDSFHFDGHFH98}, we observe that   \begin{align}   \begin{split}   &-\sum_k \int \FGSDFHGFHDFGHDFGH_k (\phi R( \rho(u \cdot \grad \eta)))   |\FGSDFHGFHDFGHDFGH_k\tzeta|^{2p-2}\FGSDFHGFHDFGHDFGH_k\tzeta   \les \sum_k \Vert \FGSDFHGFHDFGHDFGH_k (\phi R( \rho(u \cdot \grad \eta))) \Vert_{L^{2p}}   \Vert |\FGSDFHGFHDFGHDFGH_k\tzeta|^{2p-1}\Vert_{L^{2p/(2p-1)}}   \\& \indeq   \les   (   \Vert \phi \Vert_{L^\infty}\Vert \grad R( \rho(u \cdot \grad \eta))\Vert_{L^{2p}}   +   \Vert \grad \phi \Vert_{L^\infty} \Vert R( \rho(u \cdot \grad \eta))\Vert_{L^{2p}}   )   \sum_k \Vert \FGSDFHGFHDFGHDFGH_k\tzeta \Vert_{L^{2p} }^{2p-1}   \\& \indeq   \les   \Vert \rho \Vert_{L^{2p}} \Vert u \Vert_{L^\infty} \Vert \grad \eta \Vert_{L^\infty}   \sum_k\Vert \FGSDFHGFHDFGHDFGH_k\tzeta \Vert_{L^{2p} }^{2p-1}   \les   \sum_k\Vert \FGSDFHGFHDFGHDFGH_k\tzeta \Vert_{L^{2p} }^{2p-1}   ,   \end{split}   \llabel{ PY z SWj Pj5 tYZ ET Yzg6 Ix5t ATPMdl Gk e 67X b7F ktE sz yFyc mVhG JZ29aP gz k Yj4 cEr HCd P7 XFHU O9zo y4AZai SR O pIn 0tp 7kZ zU VHQt m3ip 3xEd41 By 7 2ux IiY 8BC Lb OYGo LDwp juza6i Pa k Zdh aD3 xSX yj pdOw oqQq Jl6RFg lO t X67 nm7 s1l ZJ mGUr dIdX Q7jps7 rc d ACY ZMs BKA Nx tkqf Nhkt sbBf2O BN Z 5pf oqS Xtd 3c HFLN tLgR oHrnNl wR n ylZ EWRTDFBSERSDFGSDFHDSFADSFSDFSVBASWERTDFSGSDFGFDGASFASDFDSFHFDGHFH107}   \end{align} where we used \eqref{EWRTDFBSERSDFGSDFHDSFADSFSDFSVBASWERTDFSGSDFGFDGASFASDFDSFHFDGHFH89}. For the fifth term on the right-hand side of~\eqref{EWRTDFBSERSDFGSDFHDSFADSFSDFSVBASWERTDFSGSDFGFDGASFASDFDSFHFDGHFH98}, we determine that   \begin{align}   \begin{split}   &-2\sum_k \int \FGSDFHGFHDFGHDFGH_k(\FGSDFHGFHDFGHDFGH_j \zeta \FGSDFHGFHDFGHDFGH_j \phi)|\FGSDFHGFHDFGHDFGH_k\tzeta|^{2p-2}\FGSDFHGFHDFGHDFGH_k\tzeta   =   -\frac{2p-1}{p} \sum_k \int \FGSDFHGFHDFGHDFGH_j \zeta \FGSDFHGFHDFGHDFGH_j \phi |\FGSDFHGFHDFGHDFGH_k\tzeta|^{p-2}\FGSDFHGFHDFGHDFGH_k\tzeta   \FGSDFHGFHDFGHDFGH_k(|\FGSDFHGFHDFGHDFGH_k\tzeta|^p)   \\& \indeq   \les   \Vert \FGSDFHGFHDFGHDFGH_j \zeta \FGSDFHGFHDFGHDFGH_j \phi \Vert_{L^4}   \sum_k \Vert |\FGSDFHGFHDFGHDFGH_k\tzeta|^{p-1} \Vert_{L^4}   \Vert \grad (|\FGSDFHGFHDFGHDFGH_k\tzeta|^p) \Vert_{L^2}   \les   {\bar D}^{1/2}    \Vert \nabla \zeta \Vert_{L^4}   \sum_k \Vert |\FGSDFHGFHDFGHDFGH_k\tzeta|^p \Vert_{L^{4(p-1)/p}}^{(p-1)/p}   .   \end{split}   \llabel{NWV NfH vO B1nU Ayjt xTWW4o Cq P Rtu Vua nMk Lv qbxp Ni0x YnOkcd FB d rw1 Nu7 cKy bL jCF7 P4dx j0Sbz9 fa V CWk VFo s9t 2a QIPK ORuE jEMtbS Hs Y eG5 Z7u MWW Aw RnR8 FwFC zXVVxn FU f yKL Nk4 eOI ly n3Cl I5HP 8XP6S4 KF f Il6 2Vl bXg ca uth8 61pU WUx2aQ TW g rZw cAx 52T kq oZXV g0QG rBrrpe iw u WyJ td9 ooD 8t UzAd LSnI tarmhP AW B mnm nsb xLI qXEWRTDFBSERSDFGSDFHDSFADSFSDFSVBASWERTDFSGSDFGFDGASFASDFDSFHFDGHFH108}   \end{align} By the Gagliardo-Nirenberg inequality, we have for the last factor   \begin{align}   \begin{split}   &\Vert |\FGSDFHGFHDFGHDFGH_k\tzeta|^p \Vert_{L^{4(p-1)/p}}^{(p-1)/p}   \les   \left(    \Vert |\FGSDFHGFHDFGHDFGH_k\tzeta|^p \Vert_{L^2}^{p/(2p-2)}   \Vert \grad (|\FGSDFHGFHDFGHDFGH_k\tzeta|^p) \Vert_{L^2}^{(p-2)/(2p-2)}   \right)^{(p-1)/p}   \\& \indeq   \les   \Vert |\FGSDFHGFHDFGHDFGH_k\tzeta|^p \Vert_{L^2}^{1/2}   \Vert \grad (|\FGSDFHGFHDFGHDFGH_k\tzeta|^p) \Vert_{L^2}^{(p-2)/2p}   \les   {\bar D}^{(p-2)/4p} \Vert \FGSDFHGFHDFGHDFGH_k\tzeta \Vert_{L^{2p}}^{p/2}   ,   \end{split}   \llabel{ 4RQS TyoF DIikpe IL h WZZ 8ic JGa 91 HxRb 97kn Whp9sA Vz P o85 60p RN2 PS MGMM FK5X W52OnW Iy o Yng xWn o86 8S Kbbu 1Iq1 SyPkHJ VC v seV GWr hUd ew Xw6C SY1b e3hD9P Kh a 1y0 SRw yxi AG zdCM VMmi JaemmP 8x r bJX bKL DYE 1F pXUK ADtF 9ewhNe fd 2 XRu tTl 1HY JV p5cA hM1J fK7UIc pk d TbE ndM 6FW HA 72Pg LHzX lUo39o W9 0 BuD eJS lnV Rv z8VD V48tEWRTDFBSERSDFGSDFHDSFADSFSDFSVBASWERTDFSGSDFGFDGASFASDFDSFHFDGHFH109}   \end{align} for $k=1,2$. Therefore, by Young's inequality, we conclude that   \begin{align}   \begin{split}   &-2\sum_k \int \FGSDFHGFHDFGHDFGH_k(\FGSDFHGFHDFGHDFGH_j \zeta \FGSDFHGFHDFGHDFGH_j \phi)|\FGSDFHGFHDFGHDFGH_k\tzeta|^{2p-2}\FGSDFHGFHDFGHDFGH_k\tzeta   \les   {\bar D}^{(3p-2)/4p}   \vert \nabla \zeta \vert_{L^4}   \sum_k\vert \FGSDFHGFHDFGHDFGH_k\tzeta \vert_{L^{2p}}^{p/2}   \\& \indeq   \leq \frac{{\bar D}}{8}    + C  \vert \nabla\zeta \vert_{L^4}^{4p/(p+2)}   \sum_k \vert \FGSDFHGFHDFGHDFGH_k\tzeta \vert_{L^{2p}}^{2p^2/(p+2)}   .   \end{split}   \llabel{ Id4Dtg FO O a47 LEH 8Qw nR GNBM 0RRU LluASz jx x wGI BHm Vyy Ld kGww 5eEg HFvsFU nz l 0vg OaQ DCV Ez 64r8 UvVH TtDykr Eu F aS3 5p5 yn6 QZ UcX3 mfET Exz1kv qE p OVV EFP IVp zQ lMOI Z2yT TxIUOm 0f W L1W oxC tlX Ws 9HU4 EF0I Z1WDv3 TP 4 2LN 7Tr SuR 8u Mv1t Lepv ZoeoKL xf 9 zMJ 6PU In1 S8 I4KY 13wJ TACh5X l8 O 5g0 ZGw Ddt u6 8wvr vnDC oqYjJ3 nFEWRTDFBSERSDFGSDFHDSFADSFSDFSVBASWERTDFSGSDFGFDGASFASDFDSFHFDGHFH110}   \end{align} \colb For the final term of \eqref{EWRTDFBSERSDFGSDFHDSFADSFSDFSVBASWERTDFSGSDFGFDGASFASDFDSFHFDGHFH98}, we integrate by parts and obtain   \begin{align}   \begin{split}   &\sum_k \int \FGSDFHGFHDFGHDFGH_k(\zeta (\phi_t - \Laplace \phi + u\cdot \grad \phi))   |\FGSDFHGFHDFGHDFGH_k\tzeta|^{2p-2}\FGSDFHGFHDFGHDFGH_k\tzeta   \\& \indeq   =   -\frac{2p-1}{p}   \sum_k \int \zeta (\phi_t - \Laplace \phi + u\cdot \grad \phi)   |\FGSDFHGFHDFGHDFGH_k\tzeta|^{p-2}\FGSDFHGFHDFGHDFGH_k\tzeta   \FGSDFHGFHDFGHDFGH_k(|\FGSDFHGFHDFGHDFGH_k\tzeta|^p)   \\& \indeq   \les   \Vert \zeta \Vert_{L^{2p}}    \Vert \phi_t - \Laplace \phi + u\cdot \grad \phi \Vert_{L^\infty}   \sum_k
  \Vert  |\FGSDFHGFHDFGHDFGH_k\tzeta|^{p-1} \Vert_{L^{2p/(p-1)}}   \Vert \grad (|\FGSDFHGFHDFGHDFGH_k\tzeta|^p) \Vert_{L^2}   \\& \indeq   \les   {\bar D}^{1/2}   \sum_k \Vert \FGSDFHGFHDFGHDFGH_k\tzeta \Vert_{L^{2p}}^{p-1}   ,   \end{split}   \llabel{ K WMA K8V OeG o4 DKxn EOyB wgmttc ES 8 dmT oAD 0YB Fl yGRB pBbo 8tQYBw bS X 2lc YnU 0fh At myR3 CKcU AQzzET Ng b ghH T64 KdO fL qFWu k07t DkzfQ1 dg B cw0 LSY lr7 9U 81QP qrdf H1tb8k Kn D l52 FhC j7T Xi P7GF C7HJ KfXgrP 4K O Og1 8BM 001 mJ PTpu bQr6 1JQu6o Gr 4 baj 60k zdX oD gAOX 2DBk LymrtN 6T 7 us2 Cp6 eZm 1a VJTY 8vYP OzMnsA qs 3 RL6 xHuEWRTDFBSERSDFGSDFHDSFADSFSDFSVBASWERTDFSGSDFGFDGASFASDFDSFHFDGHFH111}   \end{align} using \eqref{EWRTDFBSERSDFGSDFHDSFADSFSDFSVBASWERTDFSGSDFGFDGASFASDFDSFHFDGHFH41} and \eqref{EWRTDFBSERSDFGSDFHDSFADSFSDFSVBASWERTDFSGSDFGFDGASFASDFDSFHFDGHFH92}. Therefore, we have   \begin{align}   \begin{split}   &\sum_k \int \FGSDFHGFHDFGHDFGH_k(\zeta (\phi_t - \Laplace \phi + u\cdot \grad \phi))   |\FGSDFHGFHDFGHDFGH_k\tzeta|^{2p-2}\FGSDFHGFHDFGHDFGH_k\tzeta   \leq  \frac{ {\bar D} }{8} + C \sum_k \Vert \FGSDFHGFHDFGHDFGH_k\tzeta \Vert_{L^{2p}}^{2p-2}  .   \end{split}   \label{EWRTDFBSERSDFGSDFHDSFADSFSDFSVBASWERTDFSGSDFGFDGASFASDFDSFHFDGHFH112}   \end{align} Introducing   \begin{align}   \psi(t) = \sum_k \int | \FGSDFHGFHDFGHDFGH_k\tzeta |^{2p}   ,   \llabel{ mXN AB 5eXn ZRHa iECOaa MB w Ab1 5iF WGu cZ lU8J niDN KiPGWz q4 1 iBj 1kq bak ZF SvXq vSiR bLTriS y8 Q YOa mQU ZhO rG HYHW guPB zlAhua o5 9 RKU trF 5Kb js KseT PXhU qRgnNA LV t aw4 YJB tK9 fN 7bN9 IEwK LTYGtn Cc c 2nf Mcx 7Vo Bt 1IC5 teMH X4g3JK 4J s deo Dl1 Xgb m9 xWDg Z31P chRS1R 8W 1 hap 5Rh 6Jj yT NXSC Uscx K4275D 72 g pRW xcf AbZ Y7 ApEWRTDFBSERSDFGSDFHDSFADSFSDFSVBASWERTDFSGSDFGFDGASFASDFDSFHFDGHFH113}   \end{align} we may rewrite \eqref{EWRTDFBSERSDFGSDFHDSFADSFSDFSVBASWERTDFSGSDFGFDGASFASDFDSFHFDGHFH98} by applying \eqref{EWRTDFBSERSDFGSDFHDSFADSFSDFSVBASWERTDFSGSDFGFDGASFASDFDSFHFDGHFH99}, \eqref{EWRTDFBSERSDFGSDFHDSFADSFSDFSVBASWERTDFSGSDFGFDGASFASDFDSFHFDGHFH102}, \eqref{EWRTDFBSERSDFGSDFHDSFADSFSDFSVBASWERTDFSGSDFGFDGASFASDFDSFHFDGHFH105}--\eqref{EWRTDFBSERSDFGSDFHDSFADSFSDFSVBASWERTDFSGSDFGFDGASFASDFDSFHFDGHFH112} as   \begin{align}   \begin{split}   (1+\psi)' + \frac{{\bar D}}{2} 	   &\les						   o(1)(1+\psi)   +   (1+\psi)^{(p-1)/p}   +   (1+\psi)^{(2p-1)/2p}   +   \Vert \nabla \zeta \Vert_{L^4}^{4p/(p+2)} (1+\psi)^{p/(p+2)}   .   \end{split}   \label{EWRTDFBSERSDFGSDFHDSFADSFSDFSVBASWERTDFSGSDFGFDGASFASDFDSFHFDGHFH114}   \end{align} It may seem that the first term in \eqref{EWRTDFBSERSDFGSDFHDSFADSFSDFSVBASWERTDFSGSDFGFDGASFASDFDSFHFDGHFH114} causes an exponential increase of $\psi$, but importantly we have the property   \begin{equation}    \int_{0}^{t} (1+\psi)    \les    t    +    \Vert \grad \zeta \Vert_{L^{2p}([0,t];L^{2p})}^{2p}    \les    t    \comma t\geq0    ,    \label{EWRTDFBSERSDFGSDFHDSFADSFSDFSVBASWERTDFSGSDFGFDGASFASDFDSFHFDGHFH115}   \end{equation} where we used \eqref{EWRTDFBSERSDFGSDFHDSFADSFSDFSVBASWERTDFSGSDFGFDGASFASDFDSFHFDGHFH65} in the second step. Now we show that the inequality \eqref{EWRTDFBSERSDFGSDFHDSFADSFSDFSVBASWERTDFSGSDFGFDGASFASDFDSFHFDGHFH115} implies that the growth is algebraic. We divide the inequality  \eqref{EWRTDFBSERSDFGSDFHDSFADSFSDFSVBASWERTDFSGSDFGFDGASFASDFDSFHFDGHFH114} by  $(1+\psi)^{p/(p+2)}$, obtaining   \begin{align}   \begin{split}   (   (1+\psi)^{2/(p+2)}   )'   &\les   o(1) (1+\psi)^{2/(p+2)}   +   (1+\psi)^{(p^2-2)/(p^2+2p)}   +   (1+\psi)^{(3p-2)/(2p(p+2))}   +   \Vert \grad \zeta \Vert_{L^4}^{4p/(p+2)}   ,   \end{split}   \llabel{to 5SpT zO1dPA Vy Z JiW Clu OjO tE wxUB 7cTt EDqcAb YG d ZQZ fsQ 1At Hy xnPL 5K7D 91u03s 8K 2 0ro fZ9 w7T jx yG7q bCAh ssUZQu PK 7 xUe K7F 4HK fr CEPJ rgWH DZQpvR kO 8 Xve aSB OXS ee XV5j kgzL UTmMbo ma J fxu 8gA rnd zS IB0Y QSXv cZW8vo CO o OHy rEu GnS 2f nGEj jaLz ZIocQe gw H fSF KjW 2Lb KS nIcG 9Wnq Zya6qA YM S h2M mEA sw1 8n sJFY Anbr xZEWRTDFBSERSDFGSDFHDSFADSFSDFSVBASWERTDFSGSDFGFDGASFASDFDSFHFDGHFH117}   \end{align} which upon integration and applying Jensen's (or H\"older's) inequality yields for $t\geq0$,   \begin{align}   \begin{split}   &   (1+\psi)^{2/(p+2)}   \\&\indeq   \les   1   +   o(1) \int_0^t (1+\psi)^{2/(p+2)}   +   \int_0^t (1+\psi)^{(p^2-2)/(p(p+2))}   +   \int_0^t (1+\psi)^{(3p-2)/(2p(p+2))}   \\& \indeq \indeq   +   \int_0^t \Vert \grad \zeta \Vert_{L^4}^{4p/(p+2)}   \\& \indeq   \les   1   +   o(1) t^{1-2/(p+2)}\left( \int_0^t (1+\psi)\right)^{2/(p+2)}    +   t^{1-(p^2-2)/(p(p+2))}\left(\int_0^t (1+\psi)\right)^{(p^2-2)/(p(p+2))}    \\& \indeq \indeq   +   t^{1-(3p-2)/(2p(p+2))}\left(\int_0^t (1+\psi)\right)^{(3p-2)/(2p(p+2))}    \\& \indeq \indeq   +   \left(\int_0^t \Vert \grad \zeta \Vert_{L^4}^4\right)^{p/(p+2)}t^{1-p/(p+2)}   ,   \end{split}    \label{EWRTDFBSERSDFGSDFHDSFADSFSDFSVBASWERTDFSGSDFGFDGASFASDFDSFHFDGHFH118}   \end{align} where we also used $\psi(0)=0$ since $\phi$ vanishes in a neighborhood of $\{t=0\}$. Therefore, recalling \eqref{EWRTDFBSERSDFGSDFHDSFADSFSDFSVBASWERTDFSGSDFGFDGASFASDFDSFHFDGHFH115}, we have for $t\geq t_0$ the inequality   \begin{align}   \begin{split}   (1+\psi)^{2/(p+2)}   &\les   t   ,   \end{split}   \llabel{T45Z wB s BvK 9gS Ugy Bk 3dHq dvYU LhWgGK aM f Fk7 8mP 20m eV aQp2 NWIb 6hVBSe SV w nEq bq6 ucn X8 JLkI RJbJ EbwEYw nv L BgM 94G plc lu 2s3U m15E YAjs1G Ln h zG8 vmh ghs Qc EDE1 KnaH wtuxOg UD L BE5 9FL xIp vu KfJE UTQS EaZ6hu BC a KXr lni r1X mL KH3h VPrq ixmTkR zh 0 OGp Obo N6K LC E0Ga Udta nZ9Lvt 1K Z eN5 GQc LQL L0 PEWRTDFBSERSDFGSDFHDSFADSFSDFSVBASWERTDFSGSDFGFDGASFASDFDSFHFDGHFH119}   \end{align} where we used $\int_\delta^t \Vert \grad \zeta \Vert_{L^4}^4\les_{\delta} t$ for $\delta>0$ on the last term in \eqref{EWRTDFBSERSDFGSDFHDSFADSFSDFSVBASWERTDFSGSDFGFDGASFASDFDSFHFDGHFH118}. Raising the resulting inequality to $(p+2)/2$, we obtain   \begin{align}   \begin{split}   1+\psi \leq C_p t^{(p+2)/2}   .   \end{split}   \llabel{9GX uakH m6kqk7 qm X UVH 2bU Hga v0 Wp6Q 8JyI TzlpqW 0Y k 1fX 8gj Gci bR arme Si8l w03Win NX w 1gv vcD eDP Sa bsVw Zu4h aO1V2D qw k JoR Shj MBg ry glA9 3DBd S0mYAc El 5 aEd pII DT5 mb SVuX o8Nl Y24WCA 6d f CVF 6Al a6i Ns 7GCh OvFA hbxw9Q 71 Z RC8 yRi 1zZ dM rpt7 3dou ogkAkG GE 4 87V ii4 Ofw Je sXUR dzVL HU0zms 8W 2 Ztz iEWRTDFBSERSDFGSDFHDSFADSFSDFSVBASWERTDFSGSDFGFDGASFASDFDSFHFDGHFH120}   \end{align} By the support properties of $\phi$ and $\eta$, we get for $p\geq1$   \begin{align}   \begin{split}   \Vert \grad \omega \Vert_{L^{2p}(\Omega')}   &\les   \Vert \grad \zeta \Vert_{L^{2p}(\Omega')} + p^{3/2}   \les   \Vert \grad \tzeta \Vert_{L^{2p}} + 1   \les   \psi^{1/2p} + 1   \les    t^{(p+2)/2}   ,   \end{split}   \llabel{Y5 mw9 aB ZIwk 5WNm vNM2Hd jn e wMR 8qp 2Vv up cV4P cjOG eu35u5 cQ X NTy kfT ZXA JH UnSs 4zxf Hwf10r it J Yox Rto 5OM FP hakR gzDY Pm02mG 18 v mfV 11N n87 zS X59D E0cN 99uEUz 2r T h1F P8x jrm q2 Z7ut pdRJ 2DdYkj y9 J Yko c38 Kdu Z9 vydO wkO0 djhXSx Sv H wJo XE7 9f8 qh iBr8 KYTx OfcYYF sM y j0H vK3 ayU wt 4nA5 H76b wUqyJQEWRTDFBSERSDFGSDFHDSFADSFSDFSVBASWERTDFSGSDFGFDGASFASDFDSFHFDGHFH121}   \end{align} concluding the proof. \end{proof} \par \appendix \startnewsection{Uniform Gronwall inequalities}{seca} In the appendix, we state and prove two  Gronwall inequalities needed in the proof of Theorem~\ref{T01}. The following lemma is used to show \eqref{EWRTDFBSERSDFGSDFHDSFADSFSDFSVBASWERTDFSGSDFGFDGASFASDFDSFHFDGHFH26}. \par \cole \begin{lemma} \label{L01} Assume that $x,y\colon [0,\infty)\to[0,\infty)$ are measurable functions with $x$ differentiable, which satisfy   \begin{equation}    \dot x + y \leq C (x^2 + 1)    \label{EWRTDFBSERSDFGSDFHDSFADSFSDFSVBASWERTDFSGSDFGFDGASFASDFDSFHFDGHFH133}   \end{equation} and   \begin{equation}    x\leq C y    ,    \label{EWRTDFBSERSDFGSDFHDSFADSFSDFSVBASWERTDFSGSDFGFDGASFASDFDSFHFDGHFH141}   \end{equation} for some positive constant $C$. If    \begin{equation}    \int_{0}^{\infty} x(s)\,ds < \infty    ,    \label{EWRTDFBSERSDFGSDFHDSFADSFSDFSVBASWERTDFSGSDFGFDGASFASDFDSFHFDGHFH134}   \end{equation} then $x(t)\leq C$ for $t\geq0$ and   \begin{equation}    \lim_{t\to\infty} x(t)=0    .    \label{EWRTDFBSERSDFGSDFHDSFADSFSDFSVBASWERTDFSGSDFGFDGASFASDFDSFHFDGHFH135}   \end{equation} Moreover,   \begin{equation}    \limsup_{t\to\infty}      \int_{t}^{t+a}        y(s)\,ds      \leq C a    \label{EWRTDFBSERSDFGSDFHDSFADSFSDFSVBASWERTDFSGSDFGFDGASFASDFDSFHFDGHFH140}   \end{equation} for every $a>0$, where the constant in \eqref{EWRTDFBSERSDFGSDFHDSFADSFSDFSVBASWERTDFSGSDFGFDGASFASDFDSFHFDGHFH140} depends on the constants in \eqref{EWRTDFBSERSDFGSDFHDSFADSFSDFSVBASWERTDFSGSDFGFDGASFASDFDSFHFDGHFH133} and \eqref{EWRTDFBSERSDFGSDFHDSFADSFSDFSVBASWERTDFSGSDFGFDGASFASDFDSFHFDGHFH141}. \end{lemma} \colb \par \begin{proof}[Proof of Lemma~\ref{L01}] Let $\epsilon\in(0,1]$, and denote $b=\sqrt{\epsilon}$. Based on \eqref{EWRTDFBSERSDFGSDFHDSFADSFSDFSVBASWERTDFSGSDFGFDGASFASDFDSFHFDGHFH134}, there exists $t_0>0$ such that   \begin{equation}    \int_{t}^{t+2b} x(s) \,ds     \leq \epsilon    \comma t\geq t_0    .    \label{EWRTDFBSERSDFGSDFHDSFADSFSDFSVBASWERTDFSGSDFGFDGASFASDFDSFHFDGHFH148}   \end{equation} Integrating the inequality  $\dot x\les x^2+1$ and using \eqref{EWRTDFBSERSDFGSDFHDSFADSFSDFSVBASWERTDFSGSDFGFDGASFASDFDSFHFDGHFH148}, we obtain   \begin{equation}    x(t_2)
   \leq      e^{C \epsilon} (x(t_1)+ C b)     \les     x(t_1) + b    \label{EWRTDFBSERSDFGSDFHDSFADSFSDFSVBASWERTDFSGSDFGFDGASFASDFDSFHFDGHFH152}   \end{equation} for all $t_1$ and $t_2$ such that $t_1\leq t_2\leq t_1+2b$. By \eqref{EWRTDFBSERSDFGSDFHDSFADSFSDFSVBASWERTDFSGSDFGFDGASFASDFDSFHFDGHFH148}, for every $t\geq t_0$, there exists $\tilde t\in [t,t+b]$ such that   \begin{equation}    x(\tilde t)    \les    \frac{\epsilon}{b}     ,    \llabel{ od O u8U Gjb t6v lc xYZt 6AUx wpYr18 uO v 62v jnw FrC rf Z4nl vJuh 2SpVLO vp O lZn PTG 07V Re ixBm XBxO BzpFW5 iB I O7R Vmo GnJ u8 Axol YAxl JUrYKV Kk p aIk VCu PiD O8 IHPU ndze LPTILB P5 B qYy DLZ DZa db jcJA T644 Vp6byb 1g 4 dE7 Ydz keO YL hCRe Ommx F9zsu0 rp 8 Ajz d2v Heo 7L 5zVn L8IQ WnYATK KV 1 f14 s2J geC b3 v9UJ djNN VBINix 1q 5 oyr EWRTDFBSERSDFGSDFHDSFADSFSDFSVBASWERTDFSGSDFGFDGASFASDFDSFHFDGHFH153}   \end{equation} and thus applying \eqref{EWRTDFBSERSDFGSDFHDSFADSFSDFSVBASWERTDFSGSDFGFDGASFASDFDSFHFDGHFH152} with $t_1=\tilde t$ leads to   \begin{equation}    x(t_2)    \les     \frac{\epsilon}{b}    + b    \les    \sqrt{\epsilon}    \comma \tilde t\leq t_2 \leq t_1+2b    ,    \label{EWRTDFBSERSDFGSDFHDSFADSFSDFSVBASWERTDFSGSDFGFDGASFASDFDSFHFDGHFH154}   \end{equation} where we used $b=\sqrt{\epsilon}$ in the last step. The inequality \eqref{EWRTDFBSERSDFGSDFHDSFADSFSDFSVBASWERTDFSGSDFGFDGASFASDFDSFHFDGHFH154} holds for all $t_2\geq t_0+a$, and since $\epsilon>0$ is arbitrarily small,  \eqref{EWRTDFBSERSDFGSDFHDSFADSFSDFSVBASWERTDFSGSDFGFDGASFASDFDSFHFDGHFH135} follows. The inequality \eqref{EWRTDFBSERSDFGSDFHDSFADSFSDFSVBASWERTDFSGSDFGFDGASFASDFDSFHFDGHFH140} is obtained by integrating $y\les x^2+1$ and using~\eqref{EWRTDFBSERSDFGSDFHDSFADSFSDFSVBASWERTDFSGSDFGFDGASFASDFDSFHFDGHFH135}. \end{proof} \par The next Gronwall-type lemma is needed to establish \eqref{EWRTDFBSERSDFGSDFHDSFADSFSDFSVBASWERTDFSGSDFGFDGASFASDFDSFHFDGHFH163} and \eqref{EWRTDFBSERSDFGSDFHDSFADSFSDFSVBASWERTDFSGSDFGFDGASFASDFDSFHFDGHFH37}, which are necessary for the proofs of \eqref{EWRTDFBSERSDFGSDFHDSFADSFSDFSVBASWERTDFSGSDFGFDGASFASDFDSFHFDGHFH08} and \eqref{EWRTDFBSERSDFGSDFHDSFADSFSDFSVBASWERTDFSGSDFGFDGASFASDFDSFHFDGHFH10}. \par \cole \begin{lemma} \label{L02} Assume that $x,y\colon [0,\infty)\to[0,\infty)$ are measurable functions with $x$ differentiable, which satisfy   \begin{equation}    \dot x + y \leq \phi(t) (x + 1)    \label{EWRTDFBSERSDFGSDFHDSFADSFSDFSVBASWERTDFSGSDFGFDGASFASDFDSFHFDGHFH143}   \end{equation} and   \begin{equation}    x\leq C y    ,    \label{EWRTDFBSERSDFGSDFHDSFADSFSDFSVBASWERTDFSGSDFGFDGASFASDFDSFHFDGHFH142}   \end{equation} where $\phi\colon[0,\infty)\to[0,\infty)$ is such that $\phi(t)\leq C$ for $t\in [0,\infty)$ and $\phi(t)\to0$, as $t\to\infty$. If also $x(0)\leq C$,  then    \begin{equation}    x(t)\les 1    \comma t\in[0,\infty)    \llabel{SBM 2Xt gr v8RQ MaXk a4AN9i Ni n zfH xGp A57 uA E4jM fg6S 6eNGKv JL 3 tyH 3qw dPr x2 jFXW 2Wih pSSxDr aA 7 PXg jK6 GGl Og 5PkR d2n5 3eEx4N yG h d8Z RkO NMQ qL q4sE RG0C ssQkdZ Ua O vWr pla BOW rS wSG1 SM8I z9qkpd v0 C RMs GcZ LAz 4G k70e O7k6 df4uYn R6 T 5Du KOT say 0D awWQ vn2U OOPNqQ T7 H 4Hf iKY Jcl Rq M2g9 lcQZ cvCNBP 2B b tjv VYj ojr rhEWRTDFBSERSDFGSDFHDSFADSFSDFSVBASWERTDFSGSDFGFDGASFASDFDSFHFDGHFH53}   \end{equation} and    \begin{equation}    \lim_{t\to\infty} x(t)=0    \label{EWRTDFBSERSDFGSDFHDSFADSFSDFSVBASWERTDFSGSDFGFDGASFASDFDSFHFDGHFH145}   \end{equation} as well as   \begin{equation}    \limsup_{t\to\infty}      \int_{t}^{t+a}        y(s)\,ds     = 0     ,    \label{EWRTDFBSERSDFGSDFHDSFADSFSDFSVBASWERTDFSGSDFGFDGASFASDFDSFHFDGHFH146}   \end{equation} for every $a> 0$. \end{lemma} \colb \par \begin{proof}[Proof of Lemma~\ref{L02}] First, by the boundedness of $\phi$, we have   \begin{equation}    x(t) \les 1    \comma t\in[0,T]      ,    \llabel{ 78tW R886 ANdxeA SV P hK3 uPr QRs 6O SW1B wWM0 yNG9iB RI 7 opG CXk hZp Eo 2JNt kyYO pCY9HL 3o 7 Zu0 J9F Tz6 tZ GLn8 HAes o9umpy uc s 4l3 CA6 DCQ 0m 0llF Pbc8 z5Ad2l GN w SgA XeN HTN pw dS6e 3ila 2tlbXN 7c 1 itX aDZ Fak df Jkz7 TzaO 4kbVhn YH f Tda 9C3 WCb tw MXHW xoCC c4Ws2C UH B sNL FEf jS4 SG I4I4 hqHh 2nCaQ4 nM p nzY oYE 5fD sX hCHJ zTQOEWRTDFBSERSDFGSDFHDSFADSFSDFSVBASWERTDFSGSDFGFDGASFASDFDSFHFDGHFH161}   \end{equation} for every $T>0$, where the constant depends on $T$. Next, there exists $t_0>0$ such that   \begin{equation}    \dot x + \frac12 y     \leq \phi(t)    \comma t\geq t_0    ,    \llabel{ cbKmvE pl W Und VUo rrq iJ zRqT dIWS QBL96D FU d 64k 5gv Qh0 dj rGlw 795x V6KzhT l5 Y FtC rpy bHH 86 h3qn Lyzy ycGoqm Cb f h9h prB CQp Fe CxhU Z2oJ F3aKgQ H8 R yIm F9t Eks gP FMMJ TAIy z3ohWj Hx M R86 KJO NKT c3 uyRN nSKH lhb11Q 9C w rf8 iiX qyY L4 zh9s 8NTE ve539G zL g vhD N7F eXo 5k AWAT 6Vrw htDQwy tu H Oa5 UIO Exb Mp V2AH puuC HWItfO ruEWRTDFBSERSDFGSDFHDSFADSFSDFSVBASWERTDFSGSDFGFDGASFASDFDSFHFDGHFH144}   \end{equation} which is obtained by choosing $t_0$ so large that the term containing $x$ on the right-hand side of \eqref{EWRTDFBSERSDFGSDFHDSFADSFSDFSVBASWERTDFSGSDFGFDGASFASDFDSFHFDGHFH143} is absorbed in the half of the second term on the left-hand side, cf.~\eqref{EWRTDFBSERSDFGSDFHDSFADSFSDFSVBASWERTDFSGSDFGFDGASFASDFDSFHFDGHFH142}. \par Let $\epsilon>0$. Then there exists $t_1\geq t_0$ such that   \begin{equation}    \dot x + \frac1C x \leq \frac{\epsilon}{2}    \comma t\geq t_1    ,    \llabel{ x YfF qsa P8u fH F16C EBXK tj6ohs uv T 8BB PDN gGf KQ g6MB K2x9 jqRbHm jI U EKB Im0 bbK ac wqIX ijrF uq9906 Vy m 3Ve 1gB dMy 9i hnbA 3gBo 5aBKK5 gf J SmN eCW wOM t9 xutz wDkX IY7nNh Wd D ppZ UOq 2Ae 0a W7A6 XoIc TSLNDZ yf 2 XjB cUw eQT Zt cuXI DYsD hdAu3V MB B BKW IcF NWQ dO u3Fb c6F8 VN77Da IH E 3MZ luL YvB mN Z2wE auXX DGpeKR nw o UVB 2oMEWRTDFBSERSDFGSDFHDSFADSFSDFSVBASWERTDFSGSDFGFDGASFASDFDSFHFDGHFH149}   \end{equation} which shows that as long as $x\geq \epsilon$, we have $\dot x+(1/C)x\leq 0$, implying an exponential decay of $x$. Therefore, by increasing $t_1$, we can assume that   \begin{equation}    x(t) \leq \epsilon    \comma t\geq t_1       .    \label{EWRTDFBSERSDFGSDFHDSFADSFSDFSVBASWERTDFSGSDFGFDGASFASDFDSFHFDGHFH147}   \end{equation} Since $\epsilon>0$ was arbitrary, we obtain \eqref{EWRTDFBSERSDFGSDFHDSFADSFSDFSVBASWERTDFSGSDFGFDGASFASDFDSFHFDGHFH145}. To prove \eqref{EWRTDFBSERSDFGSDFHDSFADSFSDFSVBASWERTDFSGSDFGFDGASFASDFDSFHFDGHFH146}, note that we may assume   \begin{equation}    \dot x + \frac12 y \leq \frac{\epsilon}{2}    \comma t\geq t_1    ,    \label{EWRTDFBSERSDFGSDFHDSFADSFSDFSVBASWERTDFSGSDFGFDGASFASDFDSFHFDGHFH150}   \end{equation} by increasing $t_1$ if necessary. Integrating \eqref{EWRTDFBSERSDFGSDFHDSFADSFSDFSVBASWERTDFSGSDFGFDGASFASDFDSFHFDGHFH150} between $t$ and $t+a$, we get   \begin{equation}    \int_{t}^{t+a} y(s)\,ds    \les x(t) + \epsilon a    \les \epsilon (1+a)    \comma t\geq t_1     ,    \llabel{ VVe hW 0ejG gbgz Iw9FwQ hN Y rFI 4pT lqr Wn Xzz2 qBba lv3snl 2j a vzU Snc pwh cG J0Di 3Lr3 rs6F23 6o b LtD vN9 KqA pO uold 3sec xqgSQN ZN f w5t BGX Pdv W0 k6G4 Byh9 V3IicO nR 2 obf x3j rwt 37 u82f wxwj SmOQq0 pq 4 qfv rN4 kFW hP HRmy lxBx 1zCUhs DN Y INv Ldt VDG 35 kTMT 0ChP EdjSG4 rW N 6v5 IIM TVB 5y cWuY OoU6 Sevyec OT f ZJv BjS ZZk M6 8vEWRTDFBSERSDFGSDFHDSFADSFSDFSVBASWERTDFSGSDFGFDGASFASDFDSFHFDGHFH151}\end{equation} where we used \eqref{EWRTDFBSERSDFGSDFHDSFADSFSDFSVBASWERTDFSGSDFGFDGASFASDFDSFHFDGHFH147} in the last step. Since $\epsilon>0$ was arbitrary, we obtain \eqref{EWRTDFBSERSDFGSDFHDSFADSFSDFSVBASWERTDFSGSDFGFDGASFASDFDSFHFDGHFH146}. \end{proof} \par \section*{Acknowledgments} IK and DM were supported in part by the NSF grant DMS-1907992.  \par 
 \end{document}